\pdfoutput=1

\documentclass[12pt,draftcls,onecolumn]{IEEEtran}

\usepackage{cite}

\ifCLASSINFOpdf
\usepackage[pdftex]{graphicx}
\else
\usepackage[dvips]{graphicx}
\fi

\usepackage[cmex10]{amsmath}

\usepackage{array} 

\usepackage[font=footnotesize,caption=false]{subfig}

\usepackage{fixltx2e} 

\usepackage{url} 

\usepackage{amsmath}
\usepackage{amssymb}
\usepackage{graphicx}
\usepackage{microtype}
\usepackage{cite}

\usepackage{algorithm}
\usepackage{subfig}
\usepackage{tikz}
\usepackage{comment}

\newcommand{\vct}[1]{\boldsymbol{#1}}

\newcommand{\real}[1]{\ensuremath{\mathbb{R}^{#1}}{}}
\newcommand{\naturals}{\mathbb{N}}
\newcommand{\bmat}[1]{\begin{bmatrix}#1\end{bmatrix}}

\newcommand{\transpose}{^{T}}

\newcommand{\de}{\mathrm{d}}
\newcommand{\metric}[2]{\langle #1, #2\rangle}
\newcommand{\inverse}{^{-1}}
\newcommand{\dert}{\frac{\de}{\de t}}
\providecommand{\abs}[1]{\lvert#1\rvert}
\providecommand{\norm}[1]{\lVert#1\rVert}

\newcommand{\xxij}{{\vct{x}_{ij}}} 
 
\newcommand{\pij}{\phi_{ij}}

\newcommand{\ggij}{\vct{\gamma}_{ij}}

\newcommand{\M}{\mathcal{M}}
\newcommand{\calI}{\mathcal{I}}
\newcommand{\calB}{\mathcal{B}}

\newcommand{\calD}{\mathcal{D}}
\newcommand{\calE}{\mathcal{E}}
\newcommand{\Sset}{\mathcal{S}}
\newcommand{\calU}{\mathcal{U}}
\newcommand{\calV}{\mathcal{V}}
\newcommand{\X}{\mathcal{X}}

\newcommand{\kcurv}{\kappa}

\newcommand{\Sphere}[1]{\mathbb{S}^{#1}}
\newcommand{\Grassmann}{\mathrm{Grass}}
\newcommand{\Deg}{\mathrm{deg}(G)}

\newtheorem{thm}{Theorem}
\newtheorem{definition}[thm]{Definition}
\newtheorem{proposition}[thm]{Proposition}
\newtheorem{lemma}[thm]{Lemma}
\newtheorem{corollary}[thm]{Corollary}

\DeclareMathOperator*{\diag}{diag}
\DeclareMathOperator*{\argmin}{argmin}
\DeclareMathOperator*{\trace}{tr}
\DeclareMathOperator*{\inj}{inj}
\DeclareMathOperator{\grad}{{grad}}
\DeclareMathOperator{\hull}{hull}

\DeclareMathOperator{\diam}{diam}

\DeclareMathOperator{\sd}{s_\Delta}
\DeclareMathOperator{\cd}{c_\Delta}

\DeclareMathOperator{\sk}{s_\kappa}
\DeclareMathOperator{\ck}{c_\kappa}
\DeclareMathOperator{\shk}{sh_\kappa}
\DeclareMathOperator{\chk}{ch_\kappa}
\DeclareMathOperator{\hess}{Hess}
\newcommand{\sqd}{\sqrt{\Delta}}

\newcommand{\eg}{{e.g.,~}}
\newcommand{\ie}{{i.e.,~}}

\newcommand{\myparagraph}[1]{\smallskip\noindent\textbf{\emph{#1.}}}

\pgfrealjobname{TAC11-manifolds}

\begin{document}
%
\title{Riemannian Consensus for Manifolds with Bounded Curvature}
%
%
%

\author{Roberto~Tron,~\IEEEmembership{Student Member,~IEEE,}
        Bijan Afsari,
        and~Ren\'e Vidal,~\IEEEmembership{Senior Member,~IEEE,}
\thanks{This work was supported by the grant NSF  CNS-0834470}
\thanks{The authors are with the Center for Imaging Science, Johns Hopkins University, Baltimore MD, 21202, USA}}

%
%

\markboth{IEEE Transactions on Automatic Control}%
{Tron \MakeLowercase{\textit{et al.}}: Riemannian Consensus for Manifolds with Bounded Curvature}
%





\maketitle

\begin{abstract}
  Consensus algorithms are popular distributed algorithms for computing aggregate quantities,
  such as averages, in ad-hoc wireless networks. However, existing algorithms mostly address the case where the measurements lie in a Euclidean space. In this work we propose Riemannian consensus, a natural extension of the existing averaging consensus algorithm to the case of Riemannian manifolds. Unlike previous generalizations, our algorithm is intrinsic and, in principle, can be applied to any complete Riemannian manifold.
We characterize our algorithm by giving sufficient convergence conditions on Riemannian manifolds with bounded curvature and we analyze the differences that rise with respect to the classical Euclidean case.
We test the proposed algorithms on synthetic data sampled from manifolds such as the space of rotations, the sphere and the Grassmann manifold.
\end{abstract}

\IEEEpeerreviewmaketitle


\section{INTRODUCTION}
%
%
%
\IEEEPARstart{C}{onsider} a set of low-power sensors, where each sensor can collect measurements from the surrounding environment and can communicate with a subset of neighboring nodes through wireless channels. We are interested in \emph{distributed algorithms} in which each node performs some local computation via communication with a few neighboring nodes and all the nodes collaborate to reach an agreement on the global quantity of interest (\eg the average of the measurements). Natural candidates for this scenario are 
 \emph{consensus algorithms}, where  each node maintains a local estimate of the global average, which is updated with the estimates from the local neighbors.

The interesting characteristic of consensus algorithms is that they converge exponentially fast under very mild communication assumptions, even in the case of a time-varying network topology.
However, traditional consensus algorithms have been mainly studied for the case where the measurements and the state of each node lie in Euclidean spaces. 

\myparagraph{Prior work} In the last few years, there has been an increasing interest in extending consensus algorithms to data lying on manifolds. This problem arises in a number of applications, including distributed pose estimation \cite{Tron:ICDSC08}, camera sensor network localization \cite{Tron:CDC09} and satellite attitude synchronization \cite{Sarlette:TAC10}. Early works consider specific manifolds such as the sphere \cite{Olfati:CDC06} or the $N$-torus \cite{Scardovi:SCL07}. However, these approaches are not easily generalizable to other manifolds. The work of \cite{Sarlette:SJOC2009} considers the problems of consensus and balancing on the more general class of compact homogeneous manifolds. However, the approach is \emph{extrinsic}, \ie it is based on specific embeddings of the manifolds in Euclidean space (where classical Euclidean consensus can be employed) and requires the ability to project the result of Euclidean consensus onto the manifold. Since the approach is extrinsic, convergence properties for both fixed and time-varying network topologies follow directly from existing results in the Euclidean case. A similar approach is taken in \cite{Hatanaka:CDC10}, where the extrinsic approach is extended to the case where the mean is time-varying. 

To the best of our knowledge, the work of \cite{Tron:ICDSC08} is the first one to propose an \emph{intrinsic} approach, which does not depend on specific embeddings of the manifold and does not require the definition of a projection operation. Instead, it relies only on the intrinsic properties of the manifold, such as geodesic distances and exponential and logarithm maps. However, \cite{Tron:ICDSC08} focuses only on a specific manifold ($SO(3)$) and does not provide a thorough convergence analysis. Other works on distributed algorithms for data lying in manifolds include \cite{Sarlette:TAC10,Igarashi:TraCST09}, which address the problem of coordination on Lie groups, and \cite{Tron:CDC09}, which addresses the problem of camera localization. However, these works do not apply to the case of general manifolds, as we consider in this paper.

\myparagraph{Paper contributions}
In this paper, we propose a natural extension of consensus algorithms to measurements lying in a Riemannian manifold for the case where the network topology is fixed. We define a cost function which is the natural equivalent to the one used to derive averaging consensus in the Euclidean case.
We then obtain our Riemannian consensus algorithm as an application of Riemannian gradient descent to this cost function. This requires only the ability to compute the exponential and logarithm maps for the manifold of interest. We derive sufficient conditions for the convergence of the proposed algorithms to a consensus configuration (i.e., where all the nodes converge to the same estimate). We also point out analogies and differences with respect to the Euclidean case. 

Our work has several important contributions with respect to the state of the art. 
First, our formulation is completely intrinsic, in the sense that it is not tied to a specific embedding of the manifold. 
Second, we consider more general (complete and not necessarily compact) Riemannian manifolds. 
Third, we provide sufficient conditions for the local and, in special cases, global convergence to the sub-manifold of consensus configurations. These conditions depend on the network connectivity, the geometric configuration of the measurements and the curvature of the manifold. We also provide stronger results that hold when additional assumptions on the manifold and network connectivity are made.
Finally, we show that, while Euclidean consensus converges to the Euclidean mean of the initial measurements, the Riemannian extension does not converge to the Fr\'echet mean, which is the Riemannian equivalent of the Euclidean mean.


\myparagraph{Paper outline}
In \S\ref{sc:background} we review Euclidean consensus and relevant notions from Riemannian geometry and optimization. In \S\ref{sc:consensus} we describe our extension of consensus algorithms to data in manifolds. Our main contributions are presented in \S\ref{sc:convergence} and \S\ref{sec:convergencePoint}. We first give convergence results for the case of general manifolds. We then strengthen our results for the particular case of manifolds with constant, non-negative curvature. In \S\ref{sc:experiments} we test the proposed algorithm on manifolds such as the special orthogonal group, the sphere and the Grassmann manifold. In the Appendix we report all the additional derivations and proofs that support the claims stated in the paper.

\section{MATHEMATICAL BACKGROUND}
\label{sc:background}

In this section, we review some basic concepts related to Euclidean consensus, Riemannian geometry and optimization that are relevant to our development in the rest of the paper.


\subsection{Review of Euclidean consensus}
\label{sc:Euclideanconsensus}

Consider a network with $N$ nodes. We represent the network as a connected, undirected graph $G = (V,E)$. The vertices $i \in V = \{1,\dots,N\}$
represent the nodes of the network while the edges $\{i,j\} \in E \subseteq V\times V$
represent the communication links between nodes $i$ and $j$.
The set of neighbors of node $i$ is denoted as $N_i =
\{ j \in V |\; \{i,j\} \in E\}$ and the number of neighbors or \emph{degree} of
node $i$ as $|N_i|$. The maximum degree of the
graph $G$ is denoted as $\Deg=\max_{i} \{\abs{N_i}\}$. 

Assume that each node measures a scalar quantity $u_i\in\real{}$, $i\in V$. The goal is to obtain a distributed algorithm to compute the
average of these measurements $\bar{u} = \frac{1}{N}\sum_{i=1}^N u_i$, which is a global quantity (in the sense that involves information from all the nodes).
The well-known average consensus algorithm, to which we will refer as Euclidean consensus,
computes the average $\bar u$ by iterating the difference equation 
\begin{equation}
  \label{eq:Euclidean-protocol}
  x_i(k+1) = x_i(k) + \varepsilon \sum_{j\in N_{i}} (x_j(k) - x_i(k)),
  \; x_i(0) = u_i, 
\end{equation}
where $x_{i}(k)$ is the state of node $i$ at iteration $l$ and 
$\varepsilon\leq\frac{1}{\Deg}$ 
is the step-size.
It is easy to verify that the mean of the states is preserved at each iteration, \ie
\begin{equation}
\frac{1}{N}\sum_{i=1}^N x_i(k) = \frac{1}{N}\sum_{i=1}^N x_i(k+1) = \bar{u}.
\label{eq:Euclidean-meanpreserving}
\end{equation}
It is also easy to see that \eqref{eq:Euclidean-protocol} is in fact a gradient
descent algorithm that minimizes the function
\begin{equation}
\label{eq:Euclidean-costfunction}
\varphi(\vct{x}) =\frac{1}{2}\sum_{\{i,j\} \in E} (x_i-x_j)^2.
\end{equation}
where $\vct{x}=(x_1,\ldots, x_N)$ denotes the vectors of all states in the network.
From now on, we will use bold letters to denote $N$-tuples in which each element belongs to \real{} or another manifold $\M$. 
The cost \eqref{eq:Euclidean-costfunction} is convex and its global minimuma are achieved
when the nodes reach a \emph{consensus configuration}, \ie when $x_i=y$ for all $i\in V$ and for any $y \in \real{}$. It can be shown that with the initial conditions stated in \eqref{eq:Euclidean-protocol} and when the graph $G$ is connected, we have $\lim_{k\rightarrow\infty} x_i(k) = \bar u$, for all $i\in V$ (see, \eg \cite{Olshevsky:SOC09}). That is, all the states converge to a unique global minimizer which corresponds to the average of the initial measurements. 

In addition, notice that the average consensus algorithm can be easily extended to multivariate data $\vct{u}_i\in\real{D}$ by applying the scalar algorithm to each coordinate of $\vct{u}_{i}$. It can also be extended to situations where the network topology changes over time \cite{Olfati:TAC04}.

\subsection{Review of concepts from Riemannian geometry}
\label{sec:reviewRiemanniangeometry}
In this section we present our notation for the Riemannian geometry concepts used throughout the paper. We refer the reader to \cite{Sakai:book96,Chavel:book06} and \cite{DoCarmo:riemannian92} for further details.

Let $(\M,\metric{}{})$ be a Riemannian manifold with metric $\metric{}{}$. The \emph{tangent space} of $\M$ at a point $x \in \M$ is denoted as $T_x\M$. 
Using the metric it is possible to define \emph{geodesic} curves, which are the generalization of straight lines in $\M$.
For the remainder of the paper, we will always assume that $\M$ is \emph{geodesically complete}, \ie there always exists a minimal length geodesic between any two points in $x,y \in \M$. The length of this geodesic is said to be the \emph{distance} between the two points, and is denoted as $d(x,y)$. Most of the manifolds of practical interest are complete.

Let $v$ be a unit-length tangent vector in $T_x\M$, \ie $\|v\| = \metric{v}{v}^{\frac{1}{2}}=1$. We can then define the \emph{exponential map} $\exp_x: T_x\M\rightarrow \M$, which maps each tangent vector $tv \in T_x\M$ to the point in $\gamma(t) \in \M$ obtained by following the geodesic $\gamma(t)$ passing through $x$ with direction $v$ for a distance $t$.
Let $\tilde{\calI}_x\subset T_x\M$ be the maximal open set on which $\exp_x$ is a diffeomorphism and define the \emph{interior set} \cite[p.216]{Sakai:book96} as $\calI_x=\exp_x\tilde{\calI}_x$.
The exponential map is invertible on $\calI_x$ and we can define the \emph{logarithm map} $\log_x: \calI_x \to T_x\M$ as $\log_x=\exp_x^{-1}$. 
We denote an \emph{open geodesic ball} \cite[p. 70]{DoCarmo:riemannian92} of radius $r>0$ centered at $x\in\M$ as $\calB _\M (x,r)\subset\M$.
We also denote as $\inj_x\M$  the \emph{injectivity radius} of $\M$ at $x\in \M$, \ie the radius of the maximal geodesic ball centered at $x$ entirely contained in $\calI_x$ and as $\inj \M$ the infimum of $\inj_x\M$ over all points in $\M$.

Given a smooth function $f:\M \to \real{}$, and a tangent vector $v \in T_x\M$, one can define the \emph{directional derivative} of $f$ in the direction $v$ at $x$ as $\left.\frac{\de}{\de t}f(\gamma(t))\right|_{t=0}$, where $\gamma(t)$ is any curve such that $\gamma(0)=x$ and $\dot\gamma(0)=v$. The \emph{gradient} of $f$ on $(\M,\langle,\rangle)$ at $x \in \M$ is defined as the unique tangent vector $\grad_x f(x) \in T_x\M$ such that, for all $v\in T_x\M$, 
\begin{equation}
 \langle\grad_x f(x), v\rangle_x=\left.\frac{\de}{\de t}f(\gamma(t))\right|_{t=0}.
\end{equation}
Intuitively, as in the Euclidean case, the gradient indicates the direction along which $f$ increases the most. A point $x \in \M$ is called a \emph{critical point} \cite{Absil:book08} of $f$ if either $\grad_x f(x)=0$, \ie it is a \emph{stationary point}, or the gradient does not exist.
In this paper, we will mainly need the gradient of the squared distance function, which is given by:
\begin{equation}
\frac{1}{2}\grad_x d^2(x,y)=-\log_x(y).
\end{equation}

Given a point $x\in \M$, we denote the sectional curvature of $\sigma$, a two-dimensional subspace in $T_x\M$, as $K_\sigma(x)$.
From now on we will assume that the sectional curvature of the manifold $\M$ is bounded above by $\Delta$ and below by $\delta$. In other words, $\delta\leq K_\sigma(x) \leq \Delta$ for any point $x\in\M$ and any two-dimensional subspace $\sigma \subset T_x\M$. If $\delta=\Delta=\kcurv$, then $\M$ is said to be of constant curvature $\kcurv$. Related to the curvature and injectivity radius, we define $r^\ast>0$ as 
\begin{equation}
r^\ast=\frac{1}{2}\min\bigl\{\inj \M, \frac{\pi}{\sqrt{\Delta}}\bigr\},
\label{eq:rstar}
\end{equation}
where we use the convention that, if $\Delta\leq0$, $\frac{1}{\sqrt{\Delta}}=+\infty$. Note that any ball with radius $r\leq r^\ast$ is guaranteed to be convex \cite{DoCarmo:riemannian92}.
In addition, for the sake of brevity, we define the functions 
\begin{align}
 S_\kcurv(t)=\begin{cases} \frac{\sin(\sqrt{\kcurv}t)}{\sqrt{\kcurv}}  & \kcurv>0 \\ t & \kcurv=0 \\ \frac{ \sinh(\sqrt{\abs{\kcurv}}t)}{\sqrt{\abs{\kcurv}}} & \kcurv<0  \end{cases},\qquad
 C_\kcurv(t)=\begin{cases} \cos(\sqrt{\kcurv}t) & \kcurv>0 \\ 1 & \kcurv=0 \\ \cosh(\sqrt{\abs{\kcurv}}t) & \kcurv<0\end{cases}.
\label{eq:SkCk}
\end{align}

In the following, we will make also use of the product manifold $\M^N=\M\times \ldots \times \M$, which is the $N$-fold cartesian product of $\M$ with itself. We will use the notation $\vct{x}=(x_1,\ldots,x_N)$ to indicate a point in $\M^N$ and $\vct{v}=(v_1,\ldots,v_N) \in T_x\M$ to indicate a tangent vector. We will use the natural metric $\metric{\vct{v}}{\vct{w}}=\sum_{i=1}^N \metric{v_i}{w_i}$. As a consequence, geodesics, exponential maps, and gradients can be easily obtained by using the respective definitions on each copy of $\M$ in $\M^N$. This notation will be used when stating results that involve the states of all the nodes.

\subsection{Examples of manifolds}
We will use the following manifolds as examples throughout the paper.

\myparagraph{Euclidean space} The usual Euclidean space $\real{n}$ can be interpreted as the simplest Riemannian manifold, where the tangent space of a point is a copy of $\real{n}\!$, the metric is the usual inner product, and geodesics are straight lines. It has constant curvature $\delta=\Delta=0$ and injectivity radius $+\infty$.

\myparagraph{The orthogonal and special orthogonal groups} The $n$-dimensional orthogonal group is defined as $O(n)=\{R\in \real{n\times n}: R\transpose R = I\}$. This is the group of orthogonal $n\times n$ matrices. This group has two connected components. One of them is the \emph{special orthogonal group} $SO(n)$, which has the additional property $\det(R)=1$, and essentially describes all possible rotations in the $n$-dimensional Euclidean space. The Lie algebra for the group is $\mathfrak{so}(n)$, the space of $n\times n$ skew-symmetric matrices. The Riemannian metric at the identity is given by $\langle v_1, v_2 \rangle = \frac{1}{2}\trace(v_1\transpose v_2)$, $v_1,v_2 \in \mathfrak{so}(n)$. In this metric, the curvature bounds are $\Delta = \frac{1}{2}$, and $\delta=0$, except when $n=3$, for which the curvature is constant $\delta=\Delta=\frac{1}{4}$.  Also, the injectivity radius is $\pi$ and $r^\ast=\frac{\pi}{2}$.

\myparagraph{The Grassmann manifold} The $(n,p)$ Grassmann manifold $\Grassmann(n,p)$ is the space of $p$-dimensional subspaces in $\real{n}$. It can also be viewed as a quotient space $O(n)/\bigl(O(p)\times O(n-p)\bigr)$, which provides a Riemannian structure for it through immersion in $O(n)$ \cite{Edelman98}. The curvature bounds are $\Delta = 2$, and $\delta=0$. The injectivity radius is $\frac{\pi}{2}$ and $r^\ast=\frac{\pi}{4}$.

\myparagraph{The sphere} The $n$-dimensional sphere is defined as $\Sphere{n}=\{Y \in \real{n+1}:Y\transpose Y=1\}$. The tangent space at a point $Y$ is defined as $T_Y\Sphere{n}=\{Z\in\real{n+1}:Z\transpose Y=0\}$. As metric, we use the standard inner product between vectors in $\real{n+1}$. The geodesics follow great circles and the curvature is constant $\delta=\Delta=1$.

More details about these manifolds and about the computation of the $\exp$ and $\log$ maps can be found in \cite{Edelman98} and \cite{Tron:TR11}.

\subsection{Review of Riemannian gradient descent}
\label{sc:graddescent}
Let $\varphi: \M \to \real{}$ be a smooth function defined on a Riemannian manifold $\M$. Given an initial point $x_0 \in \M$, it is possible to define a (steepest) gradient descent algorithm on Riemannian manifolds, as shown by Algorithm \ref{alg:graddescent}.

\begin{algorithm}
\textbf{Input:} An initial element $x_0 \in \M$ 
\begin{enumerate}
\item \textbf{Initialize} $x(0)=x_0$
\item \textbf{For} $l\in\naturals$, \textbf{repeat}
  \begin{enumerate}
  \item $w=-\grad_{x}\varphi(x(l))$
  \item $x(l+1)=\exp_{x(l)}(\varepsilon(l) w)$
 \end{enumerate}
\end{enumerate}
\caption{A Riemannian steepest gradient descent algorithm}
\label{alg:graddescent}
\end{algorithm}

In practice, at each iteration $l$, the algorithm moves from the current estimate $x(l)$ to a new estimate $x(l+1)$ along the geodesic in the opposite direction of the gradient with a step size $\varepsilon(l)$. It can be shown that, under some conditions on the sequence of step sizes $\{\varepsilon(l)\}$, we have $\lim_{l\to\infty}\grad_x\varphi(x(l))=0$ \cite{Udriste:book94,Absil:book08}. In addition, if all the iterates $\{x(l)\}$ stay in a compact set $\X\subset \M$, then the sequence $\{x(l)\}$ will converge to a critical point of $\varphi$.

Algorithm \ref{alg:graddescent} gives only a basic version of a gradient-based descent algorithm on Riemannian manifolds. Many variations are possible, \eg in the computation of the descent direction and of the step size, in the curve used to search for $x(l+1)$ (which does not need to be a geodesic) or in the stopping criterion. We refer to \cite{Absil:book08,Edelman98} for some examples of such variations.

\myparagraph{Choice of a fixed step size} Ideally, one could compute the step size $\varepsilon(l)$ at each iteration by employing methods based on a line search. However, it might be more efficient or necessary to employ a pre-determined fixed step size, which is maintained constant throughout all the iterations, \ie $\varepsilon(l)\equiv\varepsilon$. This happens, for instance, when the evaluation of the cost function is computationally expensive, as it is the case for the distributed optimization problems that we will encounter in the rest of the paper.
It is well known that the choice of $\varepsilon$ affects the convergence of Algorithm~\ref{alg:graddescent}. For small step sizes, the algorithm will exhibit a slow convergence. On the other hand, if the step size is too large, the algorithm might fail to converge at all. In this section we review and extend results on methods to choose a fixed step size for Algorithm \ref{alg:graddescent} which depends only on the characteristics of the cost functions and, possibly, on the initial point $x_0$.

We will say that $\varepsilon$ is an \emph{admissible} step size if it implies that $\varphi(x(l+1))<\varphi(x(l))$ for all $l\geq 0$. We will relate admissible step sizes to bounds on the maximum eigenvalue of the Hessian of the function, or, equivalently, on the second derivative of the function evaluated along geodesics. The results will be instrumental in the proofs in \S\ref{sc:consensumanifold}. The ideas in this section are fairly standard for the case when $\M=\real{n}$ (see, for instance, \cite[p. 466]{BoydVandenberghe04}), but here we review the general case where $\M$ is a manifold (see also\cite{Udriste:book94}).

We start by defining a bound for the Hessian of the cost function, $\hess\varphi(x)$ \cite[p. 142]{DoCarmo:riemannian92}.

\begin{definition}\label{def:hessianRiemannian}
Given a twice differentiable function $\varphi(x)$ defined on an open subset $\X$ of a manifold $\M$, we say that the Hessian $\hess\varphi(x)$  is uniformly bounded on $\X$ if there exists a finite, non-negative constant $\mu_{max}$ such that, for any ${x}_0 \in \X$ and any ${v}\in T_{{x}_0}\M$, the second derivative of $\varphi$ along $\gamma_{x_0}(t)=\exp_{{x}_0}\bigl(tv\bigr)$ satisfies: 
\begin{equation}
\left.\frac{\de^2}{\de t^2}\varphi(\gamma_{x_0}(t))\right|_{t=0}=\metric{{v}}{\hess\varphi({x}_0){v}}
 \leq \mu_{max} \norm{{v}}^2.
\end{equation}
\end{definition}

Then, we state the following Lemma. 
\begin{lemma}\label{lm:quadraticbound}
  Let $\tilde{\varphi}(t)$, be a twice differentiable function defined on $\tilde{\X}\subseteq \real{}$ satisfying $\ddot{\tilde{\varphi}}(t)\leq \tilde{\mu}_{max}$ for all $\varepsilon\in\tilde{\X}$ and some $\tilde{\mu}_{max}\in\real{}$. Then the we have the bound $\tilde{\varphi}(t)\leq \tilde{\varphi}(0) + \dot{\tilde{\varphi}}(0)t +\frac{1}{2}\tilde{\mu}_{max} t^2$ for all $t \in \tilde{\X}$.
\end{lemma}

This Lemma can be applied to functions obtained by evaluating $\varphi$ along geodesics.

\begin{thm}\label{thm:boundsRiemannian}
 Let $\mu_{max}$ be a uniform bound on the Hessian $\hess(\varphi)$ as in Definition \ref{def:hessianRiemannian}. Assume $\gamma_{x_0}(t)=\exp_{{x}_0}\bigl(-t\grad_{{x}}\varphi(x_0)\bigr) \in \X$ for all $t\in (0,2 \mu_{max}^{-1})$ and let $\tilde{\varphi}(t)=\varphi \bigl (\gamma_{x_0}(t)\bigr)$. Then $\tilde{\varphi}(t)\leq\tilde{\varphi}(0)$ 
for $t \in (0,2 \mu_{max}^{-1})$, with equality if and only if ${x}_0 \in \X$ is a stationary point of $\varphi$. 
\end{thm}

The proofs of Lemma~\ref{lm:quadraticbound} and Theorem~\ref{thm:boundsRiemannian} are left as an exercise to the reader.

In the context of Algorithm~\ref{alg:graddescent}, Theorem~\ref{thm:boundsRiemannian} implies that, as long as $\varepsilon \in (0,2 \mu_{max}^{-1})$ and $x(l+1)=\gamma_{x(l)}(\varepsilon)\in\X$, the cost function is reduced at every iteration. 
However, we stress here the fact that neither Theorem~\ref{thm:boundsRiemannian}, nor Algorithm \ref{alg:graddescent}, imply that each new iterate $x(l+1)$ will belong to $\X$ when $x(l)\in\X$. 
Therefore, additional considerations are needed in order to derive complete results for the convergence of Algorithm~\ref{alg:graddescent} to a stationary point (see \S\ref{sec:convergenceset} and \S\ref{sec:convergencePoint}).

\subsection{Fr\'echet mean}
\label{sc:globalFrechet}
In order to compare the consensus algorithm that we will propose to Euclidean consensus, we will need to generalize the concept of empirical mean to data lying in Riemannian manifolds. Let $\{u_i\}_{i=1}^N$ be a set of points in a Riemannian manifold $\M$. Similarly to the geometric definition of empirical mean in the Euclidean case, we will define the \emph{Fr\'echet mean} $\bar{u}$ of the set of points as the global minimizer of the sum of squared geodesic distances, \ie
\begin{equation}
\bar{u}=\argmin_{u\in\M}\sum_{i=1}^N d^2(u_i,u).
\end{equation}
If the points lie in a ball of radius smaller than $r^\ast$, the global minimizer is unique and belongs to the same ball \cite{Afsari:ProcAMS11}. Moreover, for spaces of constant curvature the Fr\'echet mean belongs to the closed convex hull of the measurements (see \cite{Afsari:ProcAMS11} and also \S\ref{sec:convergencePoint}). 

Note that Algorithm \ref{alg:graddescent} can be used for the computation of the Fr\'echet mean $\bar{u}$. In this case, the negative gradient is $w = \frac{1}{N} \sum_{i=1}^N \log_{\bar u}(u_i)$, which is essentially a mean in $T_{\bar u}\M$. 

The conditions for the convergence to $\bar{u}$ (as opposed to other critical points) are, for the general case, only partially known \cite{Afsari:11}. These conditions depend on the spread of the points $\{u_i\}$, the step size $\varepsilon$ and the initialization $x_0$ of the algorithm.

\section{RIEMANNIAN CONSENSUS}
\label{sc:consensus}
\label{sc:consensumanifold}
In this section we present our proposed algorithm, which we call Riemannian consensus. This algorithm can be considered as a direct extension of the Euclidean consensus to the Riemannian case. The basic idea is to use the formulation of consensus as an optimization problem and define a potential function equivalent to the cost in \eqref{eq:Euclidean-costfunction} on the Riemannian manifold of interest. Riemannian gradient descent is then applied to obtain the update rules for each node.

Following the notation introduced in \S\ref{sc:Euclideanconsensus}, let us denote the measurement and the state at node $i$ as $u_i\in\M$ and $x_i \in \M$, respectively. By a straightforward generalization of the Euclidean case in \eqref{eq:Euclidean-costfunction}, we define the potential function $\varphi$ as 
\begin{equation}
\label{eq:Riemann-cost}
\varphi(\vct{x}) = \frac{1}{2}\sum_{\{i,j\}\in E} d^2(x_i,x_j). 
\end{equation}

Notice that the gradient of $\varphi$  with respect to the $i$-th element can be explicitly calculated as
\begin{align}
\grad_{x_i}\varphi \!=\!
\frac{1}{2}\!\grad_{x_i}\sum_{j \in N_i} d^2(x_i,x_j) \!=\!
- \!\sum_{j \in N_i} \log_{x_i}(x_j),
\label{eq:gradRiemann-cost}
\end{align}
where we used the facts that the graph is undirected, $d(\cdot,\cdot)$ is symmetric and $d(x_i,x_i)=0$.

Algorithm~\ref{alg:consensus-manifolds} is our first proposed consensus protocol on $\M$ and is obtained by applying Riemannian gradient descent algorithm on the cost $\varphi$.
\begin{algorithm}
\textbf{Input:} The measurements $u_i$ at each node $i\in\{1,\ldots,N\}$
  \begin{enumerate}
  \item \textbf{For} each node $i\in\{1,\ldots,N\}$ \textbf{in parallel}
      \begin{enumerate}
      \item \textbf{Initialize} the state with the local measurement, $x_i(0)=u_i$
      \item \textbf{For} $l\in\naturals$, \textbf{repeat}
        \begin{enumerate}
        \item Compute the update
          \begin{equation}
            x_{i}(l+1)=\exp_{x_i(l)}\Bigl(-\varepsilon\grad_{x_i}\varphi\bigl(x_i(l)\bigr)\Bigr)
            \label{eq:Riemann-protocol}
          \end{equation}
      \end{enumerate}
   \end{enumerate}
\end{enumerate}
\caption{Riemannian consensus}
\label{alg:consensus-manifolds}
\end{algorithm}

As mentioned before, this protocol is a natural extension of the Euclidean case. In fact, when $\M=\real{}$ with the standard metric, the updates \eqref{eq:Riemann-protocol} reduce to the standard Euclidean updates \eqref{eq:Euclidean-protocol}, by substituting exponential and logarithm maps with conventional sums and differences. However, in general, the two consensus algorithm present very different convergence properties.
On the one hand, the convergence analysis for Euclidean consensus is simple: the cost \eqref{eq:Euclidean-costfunction} is a simple quadratic function, and simple tools from optimization theory and linear algebra are sufficient. On the other hand, carrying out a similar analysis for Riemannian consensus is not trivial: the cost \eqref{eq:gradRiemann-cost} is not a simple quadratic function and we need to take into account the Riemannian geometry of the manifold. The next two sections are devoted to present our contributions to the convergence analysis of Riemannian consensus.

\section{CONVERGENCE TO THE CONSENSUS SUB-MANIFOLD}
\label{sc:convergence}

\label{sec:convergence-consensus-manifold}
In this section we analyze the convergence properties of the Riemannian consensus algorithm. 
We divide our treatment in three parts:
\begin{enumerate}
\item We notice that the cost can have multiple local minima and we define a non-zero measure subset $\Sset \subset \M^N$ that contains all \emph{global} minimizers but no other critical point (\S\ref{sec:globalminima}).
\item We give a distributed method to choose a step-size $\varepsilon$ for which the algorithm is guaranteed to reduce the cost at each iteration (\S\ref{sc:choiceepsilon}).
\item We derive sufficient conditions under which the algorithm is guaranteed to converge to the set of global minimizers, \ie to the set of consensus configurations (\S\ref{sec:convergenceset}).
\end{enumerate}
We first obtain results for general manifolds and for general network topologies. In particular we  show \emph{local} convergence to the manifold of consensus configurations, which we  refer to as \emph{consensus sub-manifold}. With additional assumptions, we  also give results on \emph{global} convergence to the same set (\S\ref{sec:convergenceGlobal}) and \emph{local} convergence to a \emph{single point} (\S\ref{sec:convergencePoint}).


\subsection{Global minimizers of the cost function $\varphi$} 
\label{sec:globalminima}
We first show that the global minimizers of $\varphi$ corresponds to consensus configurations. Let us define the consensus sub-manifold $\calD$ as the diagonal space of $\M^N$, \ie
\begin{equation}
 \calD=\left\{(y,\ldots,y) \in \M^N : y \in \M\right\}.
\end{equation}
This set represents the manifold of all possible consensus configurations of the network, where all the nodes agree on a state. The following proposition shows that the consensus sub-manifold is exactly the set of global minimizers of $\varphi$.

\smallskip
\begin{proposition}
If $G$ is connected, then $\vct{x} \in \calD$ if and only if $\vct{x}$ is a global minimizer of $\varphi$.
\label{lm:globalmincost}
\end{proposition}
\smallskip
\begin{proof}
Note that each term of $\varphi$ in \eqref{eq:Riemann-cost} is non-negative, hence $\varphi(\vct{x}) \geq 0$. Also, if $\vct{x}\in\calD$, then $\varphi(\vct{x})=0$. Thus, $\vct{x}$ is a global minimizer. Conversely, notice that $\varphi(\vct{x})=0$ implies that for each pair $\{i,j\} \in E$, we have $d^2(x_i,x_j)=0$.  By definition, $d(x_i,x_j)=0$ if and only if the points are equal, \ie $x_i=x_j$. Since $G$ is connected, $\varphi$ achieves its global minimum $\varphi(\vct{x})=0$, if and only if $x_i=x_j=y$ for any $i$ and $j$.
\end{proof}

We now define the set $\Sset \subset\M^N$ as
\begin{equation}
 \Sset=\{(x_1,\ldots,x_N) \in \M^N:
 \exists y \in \M \textrm{ for which } \max_{i\in V} d(x_i,y)<r^\ast\}.
 \label{eq:setdef}
\end{equation}
Intuitively, $\Sset$ is a tube in $\M^N$ centered around the diagonal space $\calD$ and having a ``square'' section (see Fig.~\ref{fig:constructgammaij}). Note that $\vct{x}\in\Sset$ is equivalent to saying that there exists a $y\in\M$ such that, for all $i\in\{1,\ldots,N\}$, $x_i\in\calB_\M(y,r^\ast)$. A sufficient condition for the uniqueness of the Fr\'echet mean is that $\vct{u}=(u_1,\ldots,u_n) \in \Sset$ \cite{Afsari:ProcAMS11}.

We have then the following result, which represents our first contribution.

\begin{thm}\label{th:Riemann-globalminset}
A point $\vct{x} \in \Sset$ is a critical point for $\varphi$ if and only if $\vct{x} \in \calD$. In other words, the set $\Sset$ contains all the global minima and no other critical points of $\varphi$.
\end{thm}

For the proof, we need the following Lemma, which is proven in Section \ref{sc:firstderdistance} of the Appendix.
\begin{lemma}\label{thm:derdistancepair}
 Let $x_1$, $x_2$, $y$ be three points in $\M$ such that $d(x_i,y)<r^\ast$, $i=1,2$. Define the unique minimal geodesics $\gamma_i(t)$ such that $\gamma_i(0)=y$ and $\gamma_i(1)=x_i$, $i=1,2$. Define also $\phi_{12}(t)=d(\gamma_1(t),\gamma_2(t))$. Then $\dert \phi_{12}^2(t)\geq0$ for $t\in (0,1]$, with equality if and only if $x_1=x_2$. 
\end{lemma}

\begin{proof}[Proof of Theorem \ref{th:Riemann-globalminset}]
If $\vct{x}\in\calD$, then, from Proposition~\ref{lm:globalmincost}, $\vct{x}$ is a global minimizer of $\varphi$ and hence a critical point. On the other hand, $\vct{x}\notin \calD$ cannot be a critical point of $\varphi$ because, as we will show now, there exists a geodesic $\vct{\gamma}:[0,1]\to\M^N$ such that $\vct{\gamma}(1)=\vct{x}$ and along which $\left.\frac{\de}{\de t}\varphi(\vct{\gamma}(t))\right|_{t=1}\neq 0$.
Notice that since $\vct{x} \in \Sset$, there exists a $y \in \M$ such that $\max_{i\in V} d(x_{i},y)<r^\ast$. Define unique minimal geodesics $\gamma_i(t)$ such that $\gamma_i(0)=y$ and $\gamma_i(1)=x_{i}$. Then $\vct\gamma(t)=(\gamma_1(t),\ldots,\gamma_N(t))$ is a minimal geodesic in $\M^N$ (see also Figure \ref{fig:constructgammaij}). It follows that 
\begin{equation}
 \left.\frac{\de }{\de t}\varphi(\vct\gamma(t))\right|_{t=1}=\frac{1}{2} \sum_{\{i,j\}\in E} \left.\frac{\de }{\de t}d^2\bigl(\gamma_i(t),\gamma_j(t)\bigr)\right|_{t=1}=\sum_{\{i,j\}\in E} \metric{\dot{\gamma}_i(t)}{\log_{x_i}(x_j)}+\metric{\dot{\gamma}_j(t)}{\log_{x_j}(x_i)}.
 \label{eq:dersumcost}
\end{equation}

Since $d(x_{i},y)<r^\ast$, from Lemma \ref{thm:derdistancepair} we know that each term in the sum in the RHS of~\eqref{eq:dersumcost} (i.e., each derivative) is positive except for the case where $\gamma_i(1)=\gamma_j(1)$, \ie $x_i=x_j$ and $\log_{x_i}(x_j)=0$. If \emph{all} the terms of the sum were zero, from the connectedness of $G$ we would have that $x_i=x_j$ for all $\{i,j\} \in E$, \ie $x\in \calD$. However, by assumption $\vct{x} \notin \calD$, hence at least one of the terms in \eqref{eq:dersumcost}, and therefore the entire sum $\left.\frac{\de }{\de t}\varphi(\vct\gamma(t))\right|_{t=1}$, must be strictly positive.

From the definition of gradient, $\grad_{\vct{x}}\varphi(\vct{x})=0$ if and only if the directional derivative $\left.\frac{\de }{\de t}\varphi(\vct\gamma(t))\right|_{t=1}=0$ for any curve $\vct\gamma(t)$ passing through $\vct{x}$, \ie $\vct\gamma(1)=\vct{x}$. Since we have just shown that $\left.\frac{\de }{\de t}\varphi(\vct\gamma(t))\right|_{t=1}>0$, $\vct{x}$ is not a critical point.
\end{proof}

The bounds in Lemma \ref{thm:derdistancepair} are, in general, quite conservative. In practice, there might be a set containing $\calD$ and no other critical points which is larger than $\Sset$, i.e., in general $\Sset$ is not maximal.

In fact, if the graph $G$ is a tree, we can show the following stronger result.
\begin{thm}\label{thm:minset-tree}
  If $G$ is a tree, any stationary point $\vct{x}$ of $\varphi$ is a global minimizer, \ie $\vct{x}\in \calD$.
\end{thm}
\begin{proof} We will now introduce some new notation exclusively for the purposes of this proof.
Pick an arbitrary node as the root of the tree and denote as $x^{(p)}_{i^{(p)}}$ the state of the $i^{(p)}$-th node among the ones at hop-distance $p$ from the root (\eg $x^{(0)}_1$ is the state at the root). Also, let $x^{(p-1)}_{i^{(p)}}$ and $x^{(p+1)}_{i^{(p)},j}$ denote, respectively, the parent and the $j$-th children of $x^{(p)}_{i^{(p)}}$, $j\in\{1,\abs{N_{i^{(p)}}}-1\}$. Using this notation we can rewrite \eqref{eq:gradRiemann-cost} as
\begin{equation}
  \grad_{x^{(p)}_{i^{(p)}}}\varphi=-\log_{x^{(p)}_{i^{(p)}}} x^{(p-1)}_{i^{(p)}} -\sum_{j=1}^{\abs{N_{i^{(p)}}}-1} \log_{x^{(p)}_{i^{(p)}}}x^{(p+1)}_{i^{(p)},j},
\label{eq:gradphiTree}
\end{equation}
with the appropriate modifications for the leaves and the root of $G$. Now assume $\grad_{\vct{x}}\varphi=0$. For a leaf node, \eqref{eq:gradphiTree} becomes $\log_{x^{(p)}_{i^{(p)}}} x^{(p-1)}_{i^{(p)}}=0$ (since leafs do not have any child) and therefore $x^{(p)}_{i^{(p)}}=x^{(p-1)}_{i^{(p)}}$. Now assume that, for a given hop-distance $p$, we have $x^{(p)}_{i^{(p)}}=x^{(p+1)}_{i^{(p)},j}$ for all indeces $i^{(p)}$ and $j$. Then, according to \eqref{eq:gradphiTree}, again $x^{(p)}_{i^{(p)}}=x^{(p-1)}_{i^{(p)}}$. It is then simple to show, by induction, that $x_i=x^{(0)}_1$ for any $i\in\{1,\ldots,N\}$. Therefore, $\grad_{\vct{x}}\varphi=0$ implies $\vct{x}\in\calD$.
\end{proof}

We will use Theorem~\ref{th:Riemann-globalminset} to show local convergence in general manifolds (\S\ref{sec:convergenceset}) and manifolds of non-negative, constant curvature (\S\ref{sec:convergencePoint}), while we will use Theorem~\ref{thm:minset-tree} for proving global convergence when $G$ has linear topology (\S\ref{sec:convergenceGlobal}).

\begin{figure}[b]
\begin{center}
\vspace{-3mm}

\beginpgfgraphicnamed{geodconstruction}
\begin{tikzpicture}[yscale=0.7,xscale=0.8,mypoint/.style = {fill,inner sep=1.2pt,circle},hidline/.style={thin,gray}, myarrow/.style={->,gray,shorten >= 1pt}]

 \draw (0,0) circle (2);
 \node at (-135:1.6) {$\M$};

 \node[mypoint,label=right:$y$] (y) at (0.4,-0) {};
 \node[mypoint,label=right:$x_1$] (x1) at (0.2,1.2) {};
 \node[mypoint,label=below:$x_2$] (x2) at (-0.7,-0.1) {};
 \node[mypoint,label=left:$x_3$] (x3) at (1.2,-0.8) {};

 \draw (y) circle (1.5);
 \draw (y) to[out=90,in=-90] (x1) node[pos=0.5,sloped,below]{$\gamma_1$};
 \draw (y) to[out=180,in=0] (x2) node[pos=0.5,sloped,below]{$\gamma_2$};
 \draw (y) to[out=-90,in=90] (x3) node[pos=0.8,sloped,above]{$\gamma_3$};

 \begin{scope}[xshift=5.6cm]
  \draw (0,0) ellipse (3 and 3.3);
  \node at (-45:2.7) {$\M^N$};

  \def\l{0.3}
  \draw (0,0)             +(-0.5,-2) coordinate (dl) +(0.5,2) coordinate (du);
  \draw (-1,0) ++(-\l,-\l) +(-0.5,-2) coordinate (sal) +(0.5,2) coordinate (sau);
  \draw ( 1,0) ++(-\l,-\l) +(-0.5,-2) coordinate (sbl) +(0.5,2) coordinate (sbu);
  \draw (-1,0) ++( \l, \l) +(-0.5,-2) coordinate (scl) +(0.5,2) coordinate (scu);
  \draw ( 1,0) ++( \l, \l) +(-0.5,-2) coordinate (sdl) +(0.5,2) coordinate (sdu);

  \draw[thick] (dl) to[out=90,in=-90] (du) node (vcty) [mypoint,pos=0.3,label=left:$\vct{y}$]{} node [pos=1,anchor=west]{$\calD$};

  \draw (sal) to[out=90,in=-90] (sau);
  \draw (sbl)  node[anchor=west]{$\Sset$} to[out=90,in=-90] (sbu);
  \draw[hidline] (scl) to[out=90,in=-90] (scu);
  \draw (sdl) to[out=90,in=-90] (sdu);

  \draw (sal) to[out=-10,in=170] (sbl);
  \draw (sbl) to[out=55,in=-125] (sdl);
  \draw[hidline] (scl) to[out=-10,in=170] (sdl);
  \draw[hidline] (sal) to[out=55,in=-125] (scl);

  \draw (sau) to[out=-10,in=170] (sbu);
  \draw (sbu) to[out=55,in=-125] (sdu);
  \draw (scu) to[out=-10,in=170] (sdu);
  \draw (sau) to[out=55,in=-125] (scu);

  \draw (vcty) +(0.7,0.8) node (vctx) [mypoint,label=above right:$\vct{x}$]{};
  \draw (vcty) to[out=0, in=-120] (vctx) node[pos=0.5,below,sloped] {$\vct\gamma$};
 \end{scope}

 \draw[myarrow] (y) to[out=45,in=135] (vcty);
 \draw[myarrow] (x1) to[out=45,in=125] (vctx) (x2) to[out=85,in=125] (vctx) (x3) to[out=45,in=125] (vctx);

\end{tikzpicture}
\endpgfgraphicnamed

\vspace{-3mm}

\caption{Construction of the geodesic for testing if $\varphi$ has a local minimum at $(x_i,x_j)$.}
\label{fig:constructgammaij}
\end{center}
\end{figure}
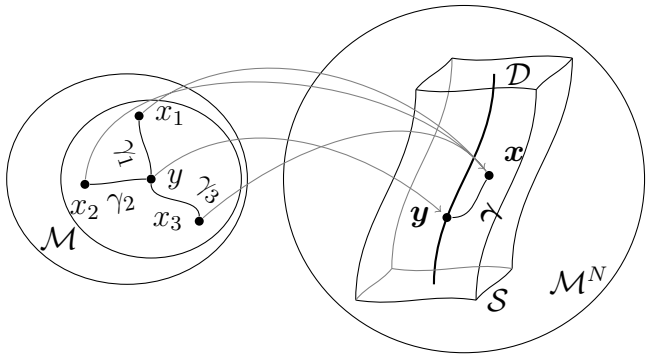

\subsection{Choice of the stepsize $\varepsilon$}
\label{sc:choiceepsilon}
In this section we provide results on the range of admissible $\varepsilon$ which can be computed in a distributed way, and guarantees convergence of the consensus protocol \eqref{eq:Riemann-protocol}.
From Theorem~\ref{thm:boundsRiemannian}, we already know that any $\varepsilon\in (0,2\mu_{max}^{-1})$ is admissible, where $\mu_{max}$
is a bound on $\hess\varphi(\vct x)$, as per Definition~\ref{def:hessianRiemannian}. However, we need to estimate a value for $\mu_{max}$, and it should be possible to compute this value in a distributed way. The following Theorem provides a step in this direction.


\begin{thm}\label{thm:boundsRiemannianSumofPairs}
  Given a graph $G=(V,E)$, let $\varphi: \M^N \to \real{}$ be a function defined as 
\begin{equation}
 \varphi(\vct{x})=\sum_{\{i,j\} \in E} \varphi_{ij}(x_i,x_j),
 \label{eq:pairwisephi}
\end{equation}
where, for all $i \in V$, $x_i \in \X_i\subseteq \M$ and, for all $\{i,j\}\in E$, $\varphi_{ij}:\X_i\times \X_j \to \real{}$. Let also $\mu_{max}^d$ be a bound on the Hessian of the pairwise function $\varphi_{ij}$, for all $\{i,j\}\in E$. Then, a bound on the Hessian of the global function $\varphi$ on $\X=\X_1 \times \ldots \times \X_N$ is given by
\begin{equation}
 \mu_{max}= \mu_{max}^d\Deg,
\end{equation}
where $\Deg$ is the maximum node degree of the graph $G$.
\end{thm}


\begin{proof}
The gradient of \eqref{eq:pairwisephi} at a point $\vct{x}=(x_{1},\ldots,x_{N})$ is given by  $\vct{v}=\grad_{\vct{x}} \varphi(\vct{x})$
where $v_i=\grad_{x_i} \varphi(\vct{x})$.
Define the cost function restricted to the geodesic along the gradient descent direction as $\tilde \varphi(t) = \varphi\bigl(\exp_{\vct{x_{0}}}(-t\vct{v})\bigr)$. Similarly, define the restriction for each pairwise term $\tilde \varphi_{ij}(t) = \varphi_{ij}\bigl(\exp_{x_{0i}} (-tv_i), \exp_{x_{j}} (-tv_j)\bigr)$.
Using the definition of $\mu_{max}^d$ we have
\begin{equation}
 \left.\frac{\de^2}{\de t^2}\tilde\varphi_{ij}(t)\right|_{t=0} < \mu_{max}^d (\norm{v_i}^2 + \norm{v_j}^2).
\end{equation}

The second derivative of $\tilde\varphi(t)$, and hence the Hessian of $\varphi$, can be uniformly bound as:
\begin{multline}
 \left.\frac{\de^2}{\de t^2}\tilde\varphi_{ij}(t)\right|_{t=0}=\sum_{\{i,j\}\in E} \frac{\de^2}{\de t^2}\tilde\varphi_{ij}(0)
 < \mu_{max}^d \sum_{\{i,j\}\in E} (\norm{v_i}^2 + \norm{v_j}^2)  = \mu_{max}^d \sum_{i \in V} \abs{N_i} \norm{v_i}^2\\
 \leq \mu_{max}^d \Deg \sum_{i\in v} \norm{v_i}^2 = \mu_{max}^d \Deg \norm{\vct{v}}^2.
\end{multline}
The claim of the Theorem follows.
\end{proof}

In our case, the global cost function $\varphi$ is given by \eqref{eq:Riemann-cost}, and $\X=\calE_{\M^N}(d_{max})$ where 
\begin{equation}
\label{eq:calE}
\calE_{\M^N}(d_{max})=\{\vct{x} \in \M^N: d(x_i,x_j)<d_{max},\,\forall \{i,j\}\in E\},
\end{equation}
and $d_{max}\leq 2r^\ast$ represents the maximum allowed distance between the states of any two neighboring nodes.
The bound on the pairwise distances $\mu_{max}^d$ is given by the following. 
\begin{thm}\label{thm:boundhessiandistance}
 The Hessian of the function $\varphi_{ij}(x_i,x_j)=\frac{1}{2}d^2(x_i,x_j)$ can be bounded on $\calE_{\M^2}$ by
\begin{equation}
\mu_{max}^d(d_{max}) = \max\left\{2,d_{max} \left(\frac{C_\delta(d_{max})}{S_\delta(d_{max})}+\frac{1}{S_\Delta(d_{max})}\right)\right\},
\end{equation}
where $C_{\kappa}(d_{max})$, $S_{\kappa}(d_{max})$ are defined in \eqref{eq:SkCk}.
\end{thm}
The proof can be found in Section \ref{sc:secondderdistancegen} of the Appendix. We remark that the bound on $\mu^d_{max}$ is sharp, in the sense that it can be achieved for manifolds with constant curvature (\ie $\delta=\Delta$, see Appendix). In fact, for Euclidean space and for spaces of non-negative constant curvature, \eg the sphere or $SO(3)$, this bound is $\mu_{max}^d=2$, and it is independent from the distance between the points. However, in general, the bound depends on $d_{max}$. Still, we might be able to find a uniform upper bound, for instance, in terms of the diameter of $\M$ or $r^\ast$. For instance, if we assume $d_{max}\leq2r^\ast$, then $\mu_{max}^d \simeq 3.792$ for both the Grassmann manifold and $SO(n)$, $n\geq4$.


From Theorem~\ref{thm:boundsRiemannianSumofPairs}, let $\mu_{max}=\Deg\mu_{max}^d(d_{max})$. We can state our second main contribution:
\smallskip

\begin{thm}\label{thm:localimprovement}
Assume that, for a given $l$, $\vct{x}(l) \in \calE_{\M^N}(d_{max})$ and $\exp_{\vct{x}(l)}\bigl(-t \grad_{\vct{x}}\varphi(\vct{x}(l))\bigr) \in \calE_{\M^N}(d_{max})$ for all $t\in \left(0,2\mu_{max}^{-1}\right)$.
If $\vct x(l+1)$ is given by the protocol \eqref{eq:Riemann-protocol} with $\varepsilon \in\left(0,2\mu_{max}^{-1}\right)$, then $\varphi(\vct x(l+1))\leq\varphi(\vct x(l))$, with equality if and only if $\vct x(l)$ is a stationary point of $\varphi$.
\end{thm}

\begin{proof}
As mentioned before, the update rule \eqref{eq:Riemann-protocol} corresponds to a Riemannian gradient descent step of $\varphi$. The claim then follows from Theorem \ref{thm:boundsRiemannian} and Proposition~\ref{thm:boundsRiemannianSumofPairs}.
\end{proof}
\smallskip

Notice that $\Sset\subseteq\calE_{\M^N}(2r^\ast)$. However, in general, $\calE_{\M^N}(2r^\ast)$ might be much larger than $\Sset$, especially when each node has a small number of neighbors. 

From Theorem \ref{thm:localimprovement}, we can deduce a simple corollary.
\begin{corollary}
 For spaces of constant curvature $\delta=\Delta\geq 0$, we can choose $\varepsilon \in (0,\Deg\inverse)$.
\end{corollary}

This tells us that the bound for the Euclidean case can be applied also for the case of manifolds with positive constant curvature (such as the sphere and $SO(3)$). In other cases (\eg for manifolds of negative curvature) we need to reduce $\varepsilon$ according to the maximum distance between the states of two neighboring nodes.

Note that we can devise distributed methods to compute a common $\varepsilon$ at each node. The maximum degree $\Deg$ can be computed in a distributed way by using a consensus-like algorithm where each node initializes its state with its own degree and repeately updates its estimate by taking the maximum of the estimates in the local neighborhood \cite{Cortes:Automatica08}. Bounds on the maximum distance can be precomputed in the case of compact manifold or, otherwise, they can  be computed in a distributed way by using consensus to estimate the value of the cost function for the measurements $\varphi(\vct{u})$ and then use ideas similar to the ones we will see in Theorem \ref{thm:globalconvergence}.

We can now establish the first result on the convergence of our consensus protocol.
\begin{thm}\label{thm:critialconvergence}
If the assumptions of Theorem \ref{thm:localimprovement} hold for any iteration $l$, then any cluster point of the sequence $\{\vct{x}(l)\} \in \calE_{\M^N}(d_{max})$ generated by \eqref{eq:Riemann-protocol} is a critical point of $\varphi$ in $\calE_{\M^N}(d_{max})$.
\end{thm}
\begin{proof}
We use a fairly standard argument. For any given iteration $l$, define $\tilde{\varphi}(t)=\varphi\bigl(\vct{\gamma}_{\vct{x(l)}}(t)\bigr)$, where $\vct{\gamma}_{\vct{x(l)}}(t)=\exp_{x(l)}\bigl(-t\grad_{\vct{x}}\varphi(\vct{x}(l))\bigr)$. Note that $\vct{x}(l+1)=\vct{\gamma}_{\vct{x}(l)}(\varepsilon)$, and $\vct{x}(l+1)\in \calE_{\M^N}(d_{max})$ by assumption. Also, $\dot{\tilde{\varphi}}(0)=-\norm{\grad_{\vct{x}} \varphi(\vct{x}(l))}^2$ and $\ddot{\tilde{\varphi}}(0)\leq\mu_{max}\norm{\grad_{\vct{x}} \varphi(\vct{x}(l))}^2$, because, by assumption $\vct{x}(l)\in \calE_{\M^N}(d_{max})$. Using Lemma 2 with $\tilde{\varphi}(t)$, we have 
\begin{equation}
 \varphi\bigl(\vct{x}(l+1)\bigr)\leq \varphi\bigl(\vct{x}(l)\bigr) - \norm{\grad_{\vct{x}} \varphi\bigl(\vct{x}(l)\bigr)}^2 \varepsilon + \frac{\mu_{max}\norm{\grad_{\vct{x}} \varphi(\vct{x}(l))}^2}{2} \varepsilon^2,
\end{equation}
when $\varepsilon$ is admissible. From this we can derive
\begin{equation}
 \varphi\bigl(\vct{x}(l)\bigr)-\varphi\bigl(\vct{x}(l+1)\bigr)
\geq    \frac{\norm{\grad_{\vct{x}} \varphi\bigl(\vct{x}(l)\bigr)}^2}{\mu_{max}}\left(c-\frac{c^2}{2}\right),
\end{equation}
where we define $c=\mu_{max}\varepsilon$.
Note that the RHS of the inequality is strictly positive, because $\varepsilon \in (0,2\mu_{max}\inverse)$ and $c\in(0,2)$. Next, since $\varphi$ is bounded below and our algorithm decreases its value at each step, we have the relation
\begin{equation}
 \frac{2c-c^2}{2\mu_{max}}\sum_{l=0}^{L}\norm{\grad_{\vct{x}} \varphi(\vct{x}(l))}^2 
\leq\!\!\sum_{l=0}^{L} \varphi(\vct{x}(l))-\varphi(\vct{x}(l+1)) 
= \varphi(\vct{x}(0))-\varphi(\vct{x}(l+1))<\infty
\end{equation}
for all $L\in\mathbb{N}$. From this argument we deduce that the series $\sum_{l=0}^{\infty}\norm{\grad_{\vct{x}} \varphi\bigl(\vct{x}(l)\bigr)}^2$ converges, $\displaystyle\lim_{l\to\infty}\norm{\grad_{\vct{x}} \varphi\bigl(\vct{x}(l)\bigr)}^2 =0$  and therefore 
the gradient vanishes
, \ie $\displaystyle \lim_{l\to\infty}\grad_{\vct{x}} \varphi(\vct{x}(l))=0$. Since $\varphi$ is continuous, this means that any cluster point of the sequence $\vct{x}(l)$ is a critical point of $\varphi$.
\end{proof}

Notice that Therorem~\ref{thm:critialconvergence} is not a complete convergence results, because it assumes that the iterates do not leave the set $\calE_{\M^N}(d_{max})$ and it does not ensure convergence to the consensus sub-manifold. These problems are going to be addressed in the next section.

\subsection{Local convergence to the consensus sub-manifold}
\label{sec:convergenceset}
This section shows that there exists a set $\Sset_{conv}\subset \Sset$ such that the algorithm converges to the set of global minimizers from any initialization in $\Sset_{conv}$. 
\begin{thm}\label{thm:globalconvergence}
 Let $D=\diam(G)$ denote the diameter of the network graph $G$ and define $\Sset_{conv}=\{\vct{x} \in \M^N: \varphi(\vct{x})<\frac{(r^\ast)^2}{2D}\}$. Then, $\Sset_{conv}\subseteq \Sset$. Moreover, if the consensus protocol \eqref{eq:Riemann-protocol} is initialized with measurements $\vct{u}\in \Sset_{conv}$ and $\varepsilon$ is admissible, then $\vct{x}(l)$ converges to $\calD$.
\end{thm}

\begin{proof}
  Consider any $p,q \in V$ and consider a shortest path in the graph $\{i_k\}_{k=0}^K$ from $i_0=p$ to $i_K=q$.  We will use this path to bound the geodesic distance between states $x_p$ and $x_q$ with the cost $\varphi$. Using the triangular and Jensen's inequalities, and the fact that $K\leq D$, we have:
\begin{multline}
 d^2(x_p,x_q)\leq \left(\sum_{k=0}^{K-1} d(x_{i_k},x_{i_{k+1}})\right)^2 \leq K\sum_{k=0}^{K-1} d^2(x_{i_k},x_{i_{k+1}})
\leq K \sum_{\{i,j\}\in E}d^2(x_i,x_j)\leq 2D\varphi(\vct{x}).
\end{multline}
This shows that if $\vct{x} \in \Sset_{conv}$, then $\varphi(\vct{x})<\frac{r^\ast}{2D}$ and $d(x_p,x_q)<r^\ast$, for any $p,q\in V$. This means that $\vct{x} \in \Sset$.
Next, we show that if $\vct{x}(l)\in \Sset_{conv}$, then $\vct{\gamma}_{\vct{x}(l)}(t)=\exp_{\vct{x}(l)}-t \grad_{\vct{x}}\varphi(\vct{x}(l)) \in \Sset$ for all $t\in \left(0,2\mu_{max}^{-1}(d_{max})\right)$. The basic idea is to show that $\vct{\gamma}_{\vct{x}(l)}$ does not cross the boundary of $\Sset_{conv}$ if $t \in \left(0,2\mu_{max}^{-1}(d_{max})\right)$. By way of contradiction, assume that there exist values of $t \in B$ such that $\varphi(\vct{\gamma}_{\vct{x}(l)}(t))=\frac{(r^\ast)^2}{2D}$ and denote as $t_0$ the minimum of such values. Then $\vct{\gamma}_{\vct{x}(l)}(t)\in \Sset_{conv}\, \forall t \in (0, t_0)$ and, since the upper bound is valid in $\Sset_{conv}$, $\varphi(\vct{\gamma}_{\vct{x}(l)}(t))<\varphi(\vct{\gamma}_{\vct{x}(l)}(0))\,\forall t\in (0, t_0)$. However, by continuity of $\varphi$, 
there must exist $\eta, \nu>0$ arbitrarily small such that $\varphi(\vct{\gamma}_{\vct{x}(l)}(t_0-\nu))\geq  \frac{(r^\ast)^2}{2D} - \eta\geq \varphi(\vct{\gamma}_{\vct{x}(l)}(0))$, which gives a contradiction.
Finally, we show that the algorithm converges to the set of global minimizers. Since $\vct{x}(0)\in \Sset_{conv}$ and $\varphi$ is decreased at each iteration, the sequence $\{\vct{x}(l)\}$ generated the protocol will be guaranteed to be in $\Sset$. From this and Theorem \ref{thm:critialconvergence}, any cluster point of the sequence $\vct{x}(l)$ will be a critical point in $\Sset$, which must be a global minimizer.
\end{proof}

Note that we have shown convergence to a set and not to a single point. Moreover, the conditions on the initial measurements depend on the size of the network. 
However, in practice, the experiments in \S\ref{sc:experiments} show convergence to a single global minimizer under much more relaxed conditions. We can give stronger versions of Theorem~\ref{thm:globalconvergence} by making additional assumptions on the manifolds and on the network topology, as we will show next.


\subsection{Special cases of global convergence to the consensus sub-manifold}
\label{sec:convergenceGlobal}
In general, the basin of attraction given by Theorem~\ref{thm:globalconvergence} can be quite small, because it depends on the diameter of the network, which might be large. Nevertheless, this condition can be relaxed for particular manifolds and network topologies.
For instance, the following is a special case for Theorem~\ref{thm:globalconvergence}.
\begin{corollary}\label{cor:rstarinfty}
  If $r^\ast=\infty$ and $\varepsilon$ is admissible, then the iterates $\vct{x}(l)$ from the consensus protocol \eqref{eq:Riemann-protocol} converge to $\calD$ for any set of initial measurements $\vct{u}$.
\end{corollary}

This corollary can be used for $\real{n}$ and some other manifolds with non-positive curvature, and it guarantees global convergence for any graph $G$.
On the other hand, if $G$ has linear topology (\ie it is a tree with a single branch), the following is true for any manifold $\M$.

\begin{thm}\label{thm:globalconvergencelinear}
 Assume $G$ has linear topology, and the consensus protocol \eqref{eq:Riemann-protocol} is initialized with measurements $\vct{u}\in\calE_{\M^N}(\inj\M)$, where $\calE$ is defined in \eqref{eq:calE}. Then $\vct{x}(l)$ converges to $\calD$. 
\end{thm}
\begin{proof}
  The assumptions imply $d(u_i,u_{i+1})<\inj\M$ for any $i\in\{1,\ldots,N-1\}$. We will now show that this same property is also satisfied by all the iterates $\vct{x}(l)$, \ie $\vct{x}(l)\in\calE_{\M^N}(\inj\M)$ for all $l\in\naturals$. For the sake of brevity, we will use the notation $d_i(l)=d(x_i(l),x_{i+1}(l))$, with the convention $d_0=d_N=0$, and $w_i(l)=\frac{\varepsilon}{2}\bigl(\log_{x_i(l)}x_{i-1}+\log_{x_i(l)}x_{i+1})\bigr)$.
By using the triangular inequality twice, we can notice that 
\begin{equation}
     d_i(l+1)=d(\exp_{x_i(l)}w_i(l),\exp_{x_{i+1}(l)}w_{i+1}(l))\leq d_i(l)+\norm{w_i(l)}+\norm{w_{i+1}(l)},
   \end{equation}
 with equality if and only if $\{x_j(l)\}_{j=i-1}^{i+2}$ all lie in order on the same geodesic. In such case, we have $\norm{w_i}=\abs{d_i(l)-d_{i+1}(l)}$ and
\begin{equation}
d_i(l+1)<d_i(l)+\norm{w_i(l)}+\norm{w_{i+1}(l)}
=(1-\varepsilon)d_i(l)-\frac{\varepsilon}{2}\bigl(d_{i-1}(l)+d_{i+1}(l)\bigr)\leq \inj\M.
\end{equation}

This shows that, at any iteration, the distance between any two neighbors will be always less than $\inj\M$, \ie $\varphi$ will always be differentiable at $\vct{x}(l)$. Combining this fact with Theorems~\ref{thm:critialconvergence} and \ref{thm:minset-tree}, we get that $\vct{x}(l)$ converges to $\calD$.
\end{proof}

\subsection{Lack of convergence to the Fr\'echet mean}

As we mentioned in \S\ref{sc:Euclideanconsensus}, when we minimize $\varphi$ in Euclidean consensus, the states converge to a global minimizer which corresponds to the average of the initial measurements.

In the Riemannian case one would expect a similar behavior, where all the states converge to the Fr\'echet mean of the measurements. However, in general this is not the case, as we will see in the experiments in \S\ref{sc:experiments}. Intuitively, this is due to the fact that the Fr\'echet mean of the states is not preserved after each iteration \cite{Tron:ICDSC08} and, even if the algorithm converges to a global minimizer (\eg under the conditions of Theorem~\ref{thm:globalconvergence}), this need not correspond to the desired Fr\'echet mean.

For computing the exact Fr\'echet mean of the measurements in a distributed way, one can extend the \emph{consensus in the tangent space} algorithm from \cite{Tron:ICDSC08} to the case of general manifolds. However, the convergence analysis of that algorithm is out of the scope of this paper.

\section{CONVERGENCE TO A SINGLE CONSENSUS CONFIGURATION FOR SPACES OF CONSTANT, NON-NEGATIVE CURVATURE}
\label{sec:convergencePoint}

In this section we prove local convergence for the specific case of spaces with constant, non-negative curvature. With this additional assumption, we can strengten Theorem~\ref{thm:globalconvergence} under three main aspects.
\begin{enumerate}
\item We enlarge the set of initializations for which convergence is guaranteed from $\Sset_{conv}$ to $\Sset$.
\item We prove convergence to a single point in the consensus sub-manifold.
\item We show that each state converges to a point in the convex hull of the initial measurements.
\end{enumerate}

In the following, we define the convex hull of a set $\calU\subset\M$, $\hull(\calU)$, as the minimal convex subset of $\M$ containing $\calU$. Regarding this definition, we will need the following Lemma.

\begin{lemma}\label{lemma:hullinclusion}
  Let $\calU,\calV\subset\M$ be two sets such that $\calU\subseteq\calV$. Then $\hull(\calU)\subseteq\hull(\calV)$.
\end{lemma}
The proof is left as an exercise to the reader.
In order to build the aforementioned stronger result, we start with the following key insight.

\begin{lemma}\label{lemma:hullneighbors}
Assume that $\M$ has constant, non-negative curvature, $\vct{x}(l)\in\Sset$, and $x_i(l+1)$ is computed according to \eqref{eq:Riemann-protocol} with $\varepsilon\in(0,\mu_{max}\inverse]$, where $\mu_{max}$ is a bound on $\hess(\varphi)$ on $\Sset$. Then $x_i(l+1)\in \hull\bigl(\{x_j(l)\}_{j \in N_i\cup \{i\}}\bigr)$.
\end{lemma}

The results follows from \cite[Theorem 5]{Afsari:11}.
We can then prove the following.
\begin{proposition}\label{prop:iterateshull}
  Assume that $\M$ has constant, non-negative curvature, $\vct{x}(l)\in\Sset$, and $x_i(l+1)$ is computed according to \eqref{eq:Riemann-protocol} with $\varepsilon\in(0,\mu_{max}\inverse]$. Then $x_i(l+1)\in \hull\bigl(\{x_i(l)\}_{i\in V}\bigr)$. Moreover, this implies $x_i(l+1)\in \hull\bigl(\{u_i\}_{i\in V}\bigr)$.
\end{proposition}
\begin{proof}
  From Lemmata~\ref{lemma:hullinclusion} and \ref{lemma:hullneighbors}, we have that $x_i(l+1)\in \hull\bigl(\{x_j(l)\}_{j \in N_i\cup \{i\}}\bigr)\subseteq\hull\bigl(\{x_i(l)\}_{i\in V}\bigr)$. The first claim follows. This also implies $\hull\bigl(\{x_i(l+1)\}_{i\in V}\bigr)\subseteq\hull\bigl(\{x_i(l)\}_{i\in V}\bigr)$ and, iteratively, $\hull\bigl(\{x_i(l+1)\}_{i\in V}\bigr)\subseteq \hull\bigl(\{x_i(0)\}_{i\in V}\bigr)= \hull\bigl(\{u_i\}_{i\in V}\bigr)$. The rest follows.
\end{proof}

We are now ready to show an improved version of Theorem~\ref{thm:globalconvergence}.
\begin{thm}\label{thm:pointconvergence}
  Assume that $\M$ has constant, non-negative curvature and $\vct{u}\in\Sset$. Then the iterates given by protocol \eqref{eq:Riemann-protocol} with $\varepsilon\in(0,\mu_{max}\inverse]$ satisfy, for all $j\in V$, $\lim_{l\to\infty}x_j(l)=y^\ast$, where $y^\ast\in\hull\bigl(\{u_i\}_{i\in V}\bigr)$.
\end{thm}
\begin{proof}
 For the sake of brevity, let $\calU=\hull\bigl(\{u_i\}_{i\in V}\bigr)$. We will show the claim in three steps.

The first step is to show that $\calU^N\subseteq\Sset$. By definition of $\Sset$, $\vct{u}\in\Sset$ implies that there exists $y\in\M$ such that $u_i\in\calB_\M(y,r^\ast)$ for all $i\in\{1,\ldots,N\}$. Hence $\calU\subset\calB_M(y,r^\ast)$. It follows that for any point say $\vct{v}$, in $\calU^N$, we also have $v_i\in\calB_M(y,r^\ast)$, which means $\vct{v}\in\Sset$. Hence $\calU^N\subseteq\Sset$.

The second step of the proof is to show that the iterates $\{\vct{x}(l)\}$ converge to a specific, bounded segment of $\calD$. From Proposition~\ref{prop:iterateshull}, we have that, for all $i\in\ V$, the sequence of iterates $\{x_i(l)\}$ remains in $\calU$. Equivalently, we have that $\vct{x}(l)\in\calU^N\subseteq\Sset$ for all $l\in\naturals$. From this fact and Theorem~\ref{thm:critialconvergence} we have therefore that the iterates $\{\vct{x}(l)\}$ converge to the set $\calD_\calU=\calD\cap\calU^N$.

The third and final step of the proof is to show convergence to a single point.
Notice that since $\vct{u}\in\Sset$, the maximum distance between any two point in $\calU$ is less than $2r^\ast$, hence $\calD_\calU$ is diffeomorphic (e.g., through the $\log$ map in $\M^N$) to a compact region in $\real{nN}$, where $n$ is the dimension of the manifold. We can then apply the Bolzano-Weierstrass theorem \cite{Protter:book98} to conclude that there exists an infinite subsequence of indeces $l_k\subset \naturals$ such that $\lim_{k\to\infty}\vct{x}(l_k)=\vct{y}^\ast\in\calD_\calU$, \ie the subsequence of iterates $\{\vct{x}(l_k)\}$ converges to a single point in $\calD_\calU$ of the form $\vct{y}^\ast=(y^\ast,\ldots,y^\ast)$ where $y^\ast\in\hull\bigl(\{u_i\}_{i\in V}\bigr)$. This implies that for any arbitrarily small $\xi\geq0$ there exists an $L\in\naturals$ large enough such that $\vct{x}(L)\in\calB_{\M^N}(\vct{y}^\ast,\xi)$. This in turn implies that $x_i(L)\in\calB_\M(y^\ast,\xi)$. Using Proposition~\ref{prop:iterateshull} we therefore get that, for all $l\geq L$, $l\in\naturals$, we have $x_i(l)\in\hull(\{x_i(l)\})\subseteq\hull(\{x_i(L)\})\subset\calB_\M(y^\ast,\xi)$.
To summarize we have that
  $\forall \xi\geq0,\, \exists L\in\naturals :\: \forall l\geq L,\,
 x_i(l)\in\calB_\M(y^\ast,\xi)$,
which, by definition, means $\lim_{l\to\infty}x_j(l)=y^\ast$.
\end{proof}

We remark that Theorem~\ref{thm:pointconvergence} is analogous to the results obtained for the centralized case in \cite{Afsari:11}.
Notice also that in Theorem~\ref{thm:pointconvergence} we require $\varepsilon\in(0,\mu_{max}\inverse]$ instead of $\varepsilon\in(0,2\mu_{max}\inverse]$, as we used to back in \S\ref{sc:convergence}. This is because in this section we rely on the fact that the iterates $\{x_i(l)\}$ never leave $\hull\bigl(\{u_i\}_{i\in V}\bigr)$, which might not be true if $\varepsilon\in(\mu_{max}\inverse,2\mu_{max}\inverse)$.
Finally, if we combine Theorem~\ref{thm:pointconvergence} with Corollary~\ref{cor:rstarinfty}, we can deduce that in the Euclidean case, where $\M=\real{n}$, the consensus algorithm has global convergence to a single consensus configuration. Indeed, this is in agreement with what we already know from the standard literature.

While we conjecture that it should be possible to extend the results of this section to the case of manifolds with non-constant curvature, extending Lemma~\ref{lemma:hullneighbors} is, in general, not trivial. Therefore, the strategy adopted here cannot be easily used to replace the results of \S\ref{sec:convergenceset}.

\section{EXPERIMENTS}
\label{sc:experiments}

In this section we evaluate the proposed algorithms on synthetic data drawn from the special orthogonal group, the sphere and the Grassmann manifold.

The experiments are performed using a synthetic network of $N=15$ nodes with a $4$-regular connectivity graph.
To generate the measurements, we choose an arbitrary element $x_0 \in \M$ and compute $N$ random tangent vectors $v_{0i}$ in $T_{x_0}\M$ drawn from an isotropic Gaussian distribution with standard deviation $\sigma=0.2$. The measurement at each node $i\in V$ is then defined as $u_i=\exp_{x_0}(v_{0i})$.
We then run our Riemannian consensus algorithm for $150$ iterations. We use step sizes compatible with the bounds found in \S\ref{sc:choiceepsilon}. After each iteration, we compute the distance between each state and the Fr\'echet mean $\bar{u}$ of the initial measurements (Figure~\ref{fig:results}, top row). We also record the distance between the Fr\'echet mean of the states at each iteration and $\bar{u}$ (Figure~\ref{fig:results}, bottom row). We have selected $SO(7)$, $\Sphere{6}$ and $\Grassmann(7,3)$ as particular examples. However, similar results are obtained on other manifolds (such as $SO(3)$).

A number of points can be made on the experiments. First, Riemannian consensus clearly converges to a single consensus configuration. This was expected, because, in this experiment, the measurements that we have generated are not too far one from the other. Second, the algorithm modifies the Fr\'echet mean of the states, especially in the first iterations. When this algorithm terminates, the estimated Fr\'echet mean is at a distance in the order of $10^{-4}$ from the true Fr\'echet mean. This error might be negligible in practical applications, but it is many order of magnitude greater than the achievable machine precision.

We include also two experiments (Fig.~\ref{fig:resultfail}) for which the measurements are taken around the circle and are far apart, \ie $\vct{u} \notin \Sset$ (see Theorem~\ref{th:Riemann-globalminset}). With a linear network, the algorithm converges to a consensus configuration, as expected from Theorem~\ref{thm:globalconvergencelinear}. On the other hand, with a ring network, the algorithm  gets trapped in a local minima and fails. These experiment suggests that the convergence of the algorithm depends on both the manifold and the network topologies. However, a complete investigation of this fact is out of the scope of this paper.

\newcommand{\figsize}{0.3\textwidth}
\begin{figure*}[t]
\subfloat[][Consensus in $SO(7)$]{\includegraphics[width=\figsize]{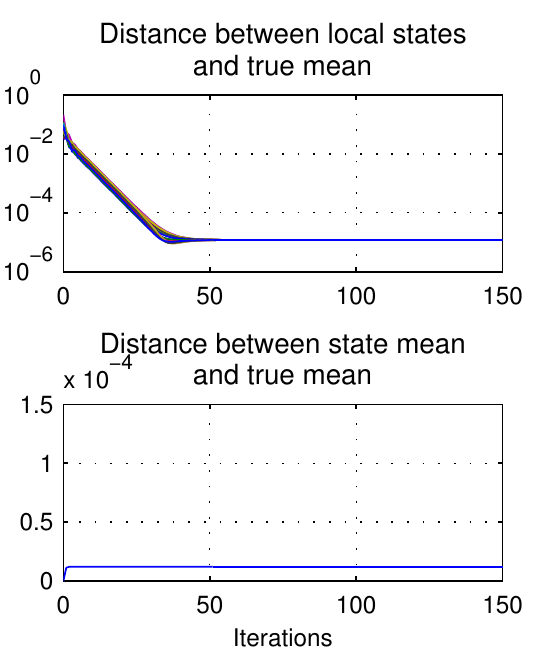} \label{fig:expRot1}} \hfill
\subfloat[][Consensus in $\Sphere{6}$]{\includegraphics[width=\figsize]{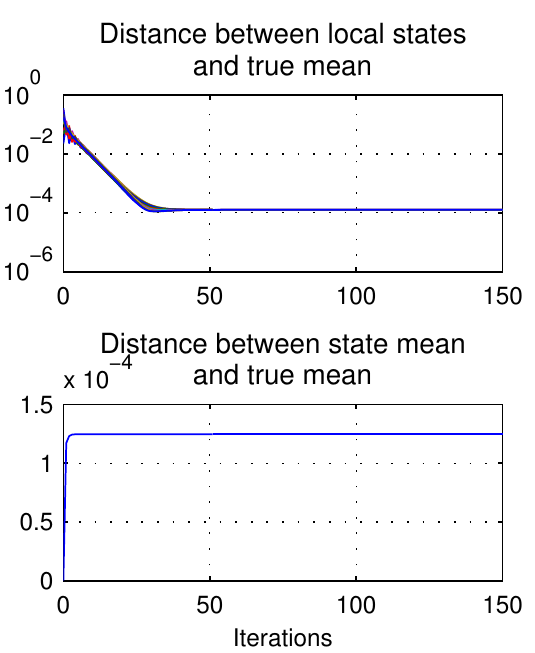} \label{fig:expSt1}} \hfill
\subfloat[][Consensus in $\Grassmann(7,3)$]{\includegraphics[width=\figsize]{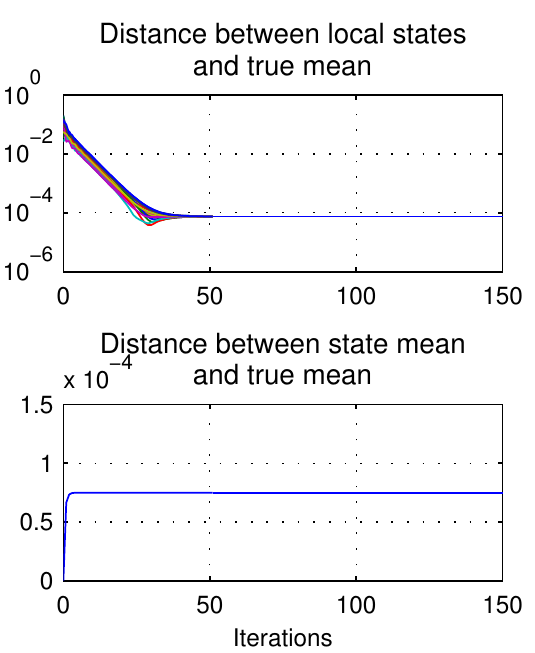} \label{fig:expGrass1}}\\
\vspace{-1mm}

\caption{Results for the algorithm applied to data in $SO(7)$, $\Sphere{6}$ and $\Grassmann(7,3)$. Top row: distances between each state and the Fr\'echet mean of the measurements for the Riemannian consensusalgorithm. Bottom row: distance between Fr\'echet mean of the states and the true Fr\'echet mean.
}
\vspace{-1mm}
\label{fig:results}
\end{figure*} 


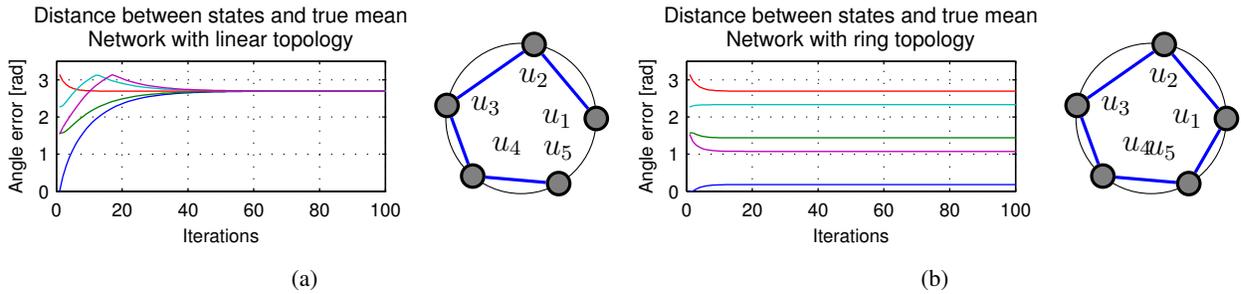
\begin{figure}
\centering
\subfloat[][]{
\includegraphics[trim=0 0 0 2mm,clip]{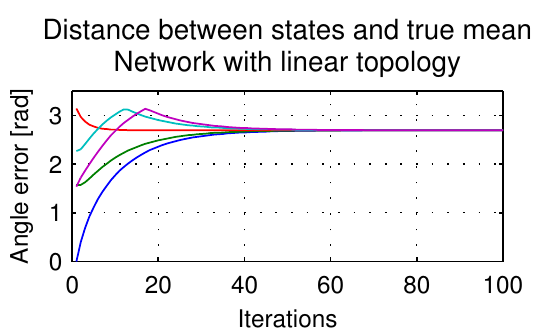}
\begin{tikzpicture}[mypoint/.style = {fill=gray,draw=black,inner sep=3pt,circle,very thick},myline/.style={very thick,blue},baseline={(0,-1.7)}]
\draw (0,0) circle (1);
\draw (0:1) node[mypoint,label=180:$u_1$](n1) {}
	(80:1) node[mypoint,label=-90:$u_2$](n2) {}
	(170:1) node[mypoint,label=0:$u_3$](n3) {}
	(230:1) node[mypoint,label=50:$u_4$](n4) {}
	(300:1) node[mypoint,label=90:$u_5$](n5) {};
\draw[myline] (n1) to (n2)
	(n2) to (n3)
	(n3) to (n4)
	(n4) to (n5);
\end{tikzpicture}
}
\subfloat[][]{
\includegraphics[trim=0 0 0 2mm,clip]{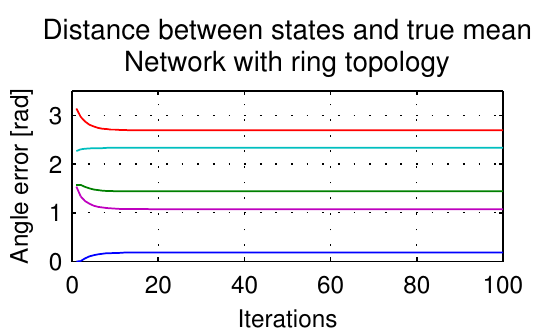}
\begin{tikzpicture}[mypoint/.style = {fill=gray,draw=black,inner sep=3pt,circle,very thick},myline/.style={very thick,blue},baseline={(0,-1.7)}]
\draw (0,0) circle (1);
\draw (0:1) node[mypoint,label=180:$u_1$](n1) {}
	(80:1) node[mypoint,label=-90:$u_2$](n2) {}
	(170:1) node[mypoint,label=0:$u_3$](n3) {}
	(230:1) node[mypoint,label=50:$u_4$](n4) {}
	(300:1) node[mypoint,label=93:$u_5$](n5) {};
\draw[myline] (n1) to (n2)
	(n2) to (n3)
	(n3) to (n4)
	(n4) to (n5)
	(n5) to (n1);
\end{tikzpicture}
}
\caption{An example where Riemannian consensus converges (a) or fails to converge (b) to a consensus configuration depending on the topology. These plots correspond to the initial configurations portrayed on the right.}
\label{fig:resultfail}
\vspace{-4mm}
\end{figure}

\section{CONCLUSIONS}
In this paper, we proposed Riemannian consensus, a natural generalization of classical consensus algorithms to Riemannian manifolds.
Our main contribution is finding sufficient conditions that guarantee convergence of the algorithm to a consensus configuration. These conditions depend on the curvature and topology of the manifold as well as the connectivity of the communication network.
Experiments on data sampled from the special orthogonal group, the sphere and the Grassmann manifold illustrated the applicability of our method.

\begin{appendix}
This appendix contains all the additional derivations and proofs for the claims in the paper.

\subsection{Additional notation}
In this section we will review additional concepts and notation from Riemannian geometry. We will focus only on those definitions and properties that are going to be applied in this Appendix. We refer the reader to standard texts (\eg \cite{DoCarmo:riemannian92,Sakai:book96}) for the complete and precise definitions.

Following the notation introduced in \S\ref{sec:reviewRiemanniangeometry}, let $(\M,\metric{}{})$ be a Riemannian manifold with its Riemannian metric.
We denote the length of a curve $\gamma: [a,b] \to \M$ between two points $x=\gamma(a)$ and $y=\gamma(b)$ as $L(\gamma)=\int_a^b \metric{\dot{\gamma}(t)}{\dot{\gamma}(t)}^\frac{1}{2} \de t$.
We denote as $\nabla$ the Levi-Civita connection on $\M$.
If $X=X(t)$ and $Y=Y(t)$ are vector fields defined along a curve $\gamma(t)$ in $\M$, then the metric compatibility property of $\nabla$ implies
$ 
  \dert\metric{X}{Y}=\metric{\nabla X}{Y}+\metric{X}{\nabla Y},
$ 
where we use the notational convention $\nabla X=\nabla_{\dot{\gamma}(t)} X$ when $X$ is a vector field along a curve. With similar notation, $X$ is said to be \emph{parallel} if $\nabla X=0$. In this case $X(t)$ is said to be the \emph{parallel transport} of $X(0)$ from $\gamma(0)$ to $\gamma(t)$ along the curve, and we use the notation $X(t)=\tau_0^tX(0)$. The curve $\gamma(t)$ is said to be geodesic if it parallel transports its own tangent, \ie $\nabla \dot{\gamma}(t)=0$.
 
The Riemannian curvature tensor $R$ is defined as
$
 R(X,Y)Z=\nabla_X\nabla_YZ-\nabla_Y\nabla_XZ-\nabla_{[X,Y]}Z,%
$ 
where $X,Y$ and $Z$ are smooth vector fields on $\M$. For the sake of clarity, we will also use the notational convention $R(X,Y,Z,W)=\metric{R(X,Y)Z}{W}$. The curvature tensor has many symmetry properties. In particular, $R(X,Y,Z,W)=-R(Y,X,Z,W)=R(Z,W,X,Y)$. Therefore, $R(X,Y,Z,W)=0$ whenever $X=Y$ or $Z=W$.  
Given a point $x\in \M$ and two linearly independent vectors $v,w \in T_x\M$ spanning a two-dimensional subspace $\sigma\subseteq T_x\M$, from the Riemannian curvature tensor one can define the sectional curvature for $\sigma$ as
$
 K_\sigma(x) = \frac{R(v,u,v,u)}{\norm{u}^2\norm{v}^2-\metric{u}{v}^2}
$. 

We denote by $\M_\kcurv$ a complete simply connected Riemannian manifold with constant curvature $\kcurv$ and with the same dimension as $\M$.
Also, 
we define the shorthand notation $\sin(\sqrt{\kcurv} x)=\sk(x)$, $\cos(\sqrt{\kcurv} x)=\ck(x)$, $\sinh(\sqrt{|\kcurv|} x)=\shk(x)$, $\cosh(\sqrt{|\kcurv|} x)=\chk(x)$.

A geodesic triangle $\triangle(x_1,x_2,x_3)$ in a Riemannian manifold $\M$ is a figure formed by three distinct points $x_1$, $x_2$ and $x_3$, called the vertices, that are connected by three minimal, unique geodesics, called the sides (see Figure~\ref{fig:deftriangle}). We denote as $\gamma_i(t)$ the side opposite to the vertex $x_i$ and we denote its length as $l_i=L(\gamma_i)$. We indicate as $\beta_i=\angle x_i=\angle(x_j,x_i,x_k)$ the oriented angle between the tangent vectors of the two geodesics emanating from $x_i$. 
A geodesic hinge $(y;\gamma_1,\gamma_2)$ in $\M$ is a figure formed by a point $y$ and two minimal geodesics segments emanating from $y$ (see Figure~\ref{fig:defhinge}).

\begin{figure}[htb]
  \vspace{-2mm}
  \hfill\subfloat[]{\beginpgfgraphicnamed{triangle}
\begin{tikzpicture}[x=1.5cm,y=1.5cm,mypoint/.style = {fill,inner sep=1.2pt,circle}, myarrow/.style={->,blue!80!black,shorten >= 1pt}, coordpoint/.style={inner sep=-1pt},angle/.style={red!50!black}]
 \node (x3) at (0,0) [mypoint,label=below:$x_3$]{};
 \node (x1) at (2.5,1.8) [mypoint,label=above:$x_1$]{};
 \node (x2) at (3,-.8) [mypoint,label=below:$x_2$]{};

 \draw (x3) to[out=20,in=-130] (x1) 
              node(abeta3a)[pos=0.21,coordpoint]{}
              node(abeta2a)[pos=0.81,coordpoint]{}
	      node(gamma1) [pos=0.5,sloped, above]{$\gamma_1$};
 \draw (x3) to[out=5,in=140] (x2) 
              node(abeta3b)[pos=0.21,coordpoint]{}
              node(abeta1a)[pos=0.79,coordpoint]{}
	      node(gamma2) [pos=0.5, sloped, below] {$\gamma_2$};
 \draw (x2) to[out=135,in=-120] (x1)
              node(abeta2b)[pos=0.79,coordpoint]{}
              node(abeta1b)[pos=0.23,coordpoint]{}
	      node(gamma3) [pos=0.5, sloped, above] {$\gamma_3$};

 \draw[angle] (abeta1a) to[out=55,in=-160] (abeta1b);
 \draw[angle] (abeta2a) to[out=-60,in=180] (abeta2b);
 \draw[angle] (abeta3a) to[out=-45,in=60] (abeta3b);

 \path (x1) +(-155:0.5) node[angle]{$\beta_1$};
 \path (x2) +(105:0.5) node[angle]{$\beta_2$};
 \path (x3) +(45:0.5) node[angle]{$\beta_3$};

\end{tikzpicture}
\endpgfgraphicnamed\label{fig:deftriangle}}\hfill
  \subfloat[]{\beginpgfgraphicnamed{hinge}
\begin{tikzpicture}[x=1.5cm,y=1.5cm,mypoint/.style = {fill,inner sep=1.2pt,circle}, myarrow/.style={->,blue!80!black,shorten >= 1pt}, coordpoint/.style={inner sep=-1pt},angle/.style={red!50!black}]
 \node (y) at (0,0) [mypoint,label=below:$y$]{};
 \node (x1) at (2.5,1.8) [mypoint]{};
 \node (x2) at (3,-.8) [mypoint]{};

 \draw (y) to[out=20,in=-130] (x1) 
	      node(gamma1) [pos=0.5,sloped, above]{$\gamma_1$};
 \draw (y) to[out=5,in=140] (x2) 
	      node(gamma2) [pos=0.5, sloped, below] {$\gamma_2$};

\end{tikzpicture}
\endpgfgraphicnamed\label{fig:defhinge}}
  \hfill{}

  \caption{Definition of \protect\subref{fig:deftriangle} geodesic
    triangle $\triangle(x_1,x_2,x_3)$ and \protect\subref{fig:defhinge} geodesic hinge $(y;\gamma_1,\gamma_2)$}
\end{figure}
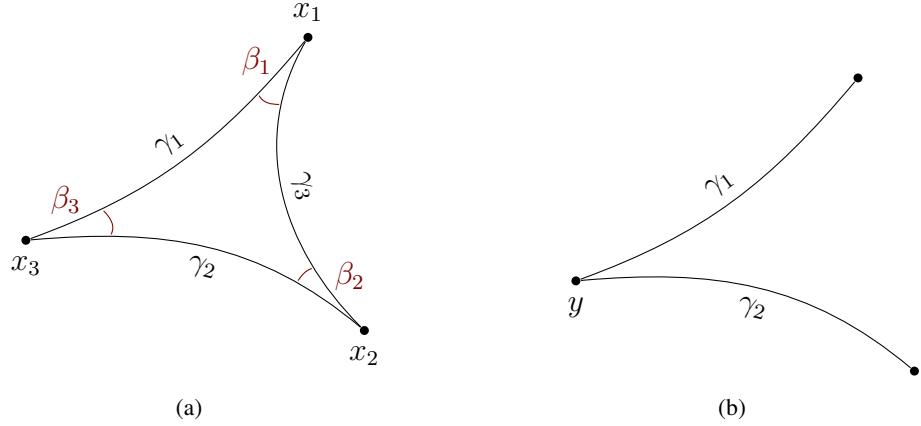

Given a vector field $X$ along a normal (\ie unit speed) geodesic $\gamma$, we define its tangential and perpendicular components as $X^\parallel=\metric{X}{\dot\gamma}\dot\gamma$ and $X^\perp=X-X^\parallel$, respectively.

A smooth vector field $Y$ along a geodesic $\gamma$ is said to be a Jacobi field if it satisfies the second order differential equation 
$\nabla\nabla Y+R(Y,\dot{\gamma})\dot{\gamma}=0$.
 Intuitively, Jacobi fields represent a variation of $\gamma$ under a perturbation of the endpoints. In fact, it is known \cite[Chapter 2, Lemma 2.4]{Sakai:book96} that a Jacobi field is uniquely determined by fixing the value of $Y$ at the two endpoints of $\gamma$. Moreover, if $Y_1$ and $Y_2$ are two Jacobi field along $\gamma$, then also $Y=Y_1+Y_2$ is a Jacobi field along $\gamma$.

\subsection{General results}
\label{sec:general-results}
In this section we collect useful results that can be easily obtained from the existing literature.

\myparagraph{Laws of cosines} In manifolds with constant curvature $\Delta$, the angles and sides of geodesic triangles are related by the \emph{laws of cosines} in Table \ref{tab:cosinelaws}.

\begin{table}[b]
\center
\begin{tabular}{c|c}
$\kappa = 0$ & $l_i^2=l_{i+1}^2+l_{i+2}^2-2l_{i+1}l_{i+2}\cos\beta_i$ \\
$\kappa >0$ & $\ck(l_i)=\ck(l_{i+1})\ck(l_{i+2})+\sk(l_{i+1})\sk(l_{i+2})\cos\beta_i$\\
$\kappa <0$ & $\chk(l_i)=\chk(l_{i+1})\chk(l_{i+2})-\shk(l_{i+1})\shk(l_{i+2})\cos\beta_i$\\
\end{tabular}
\caption{Law of cosines for geodesic triangles in manifolds of constant curvature $\Delta$}
\label{tab:cosinelaws}
\end{table}

Using these laws it is possible to show the following Lemma on geodesic triangles in manifolds with constant curvature \cite[page 138]{Sakai:book96}.
\begin{lemma}\label{thm:trianglesMD}
Let $T=\triangle(x_1,x_2,x_3)$, $T^\prime=\triangle(x_1^\prime,x_2^\prime,x_3^\prime)$ be two geodesic triangles in $\M_\Delta$. The side lengths for $T$ and $T^\prime$ are denoted as $l_i$ and $l_i^\prime$, respectively, $i=1,2,3$ and let $l_i=l_i^\prime$, $i=1,2$. If $\Delta>0$, assume also $l_2,l_3 < \pi/\sqrt{\Delta}$. Then $\angle x_i^\prime>\angle x_i$ if and only if $l_i^\prime>l_i$.
\end{lemma}

\myparagraph{Comparison theorems for geodesic triangles and hinges}
We start by reporting a hinge version of the Alexander-Toponogov theorem \cite[Exercise IX.1]{Chavel:book06}.
\begin{thm}\label{thm:ATCT1}
Given a complete Riemannian manifold $\M$ with curvature bounded above by $\Delta$ and a geodesic triangle $\triangle (x_1, x_2, x_3)$ in $\M$, assume
$l_1+l_2+l_3<2\min\left\{\inj \M, \frac{\pi}{\sqrt{\Delta}}\right\}$.
Consider the hinge $(x_3; \gamma_1,\gamma_2)$ and let $(\tilde{x}_3; \tilde{\gamma}_1, \tilde{\gamma}_2)$ be a geodesic hinge in $\M_\Delta$ such that $L(\gamma_1)=L(\tilde{\gamma}_1)$, $L(\gamma_2)=L(\tilde{\gamma}_2)$ and $\angle x_3 = \angle \tilde{x_3}$.
Then
$
d(\gamma_1(l_1),\gamma_2(l_2))\geq d(\tilde\gamma_1(l_1),\tilde\gamma_2(l_2))
$. 
\end{thm}

We will need the following triangle version of Theorem~\ref{thm:ATCT1}.
\begin{thm}\label{thm:ATCT2}
For a geodesic triangle $\triangle(x_1,x_2,x_3)$ in $\M$ suppose that $\gamma_1$ and $\gamma_2$ are minimal and the perimeter $l=l_1+l_2+l_3\leq 2\pi/\sqrt{\Delta}$. Then, there exist a geodesic triangle $\triangle(\tilde x_1, \tilde x_2, \tilde x_3)$ in $\M_\Delta$ with the same side lengths $L(\gamma_i)=L(\tilde\gamma_i)$ and satisfying $\angle x_3 \leq \angle \tilde x_3$.
\end{thm}
\begin{proof}
In addition to the triangles in $\M$ and $\M_\Delta$ defined in the statement of the theorem, define the hinge $(\tilde x_3; \tilde \gamma_1, \tilde \gamma_2^\prime)$ such that $L(\gamma_2)=L(\tilde \gamma_2^\prime)$ and $\angle x_3 = \angle \tilde x_3$ (see Figure \ref{fig:ATCT2}). Notice that $L(\gamma_1)=L(\tilde \gamma_1)$ by definition. Define $\tilde x_2^\prime = \tilde \gamma_2^\prime(l_2)$. Using Theorem \ref{thm:ATCT1} we have  
$
d(\tilde x_1, \tilde x_2)=d(x_1, x_2) \geq d(\tilde x_1, \tilde x_2^\prime)
$. 
Using Lemma \ref{thm:trianglesMD} we can obtain
$ 
\angle x_3 = \angle \tilde x_1 \tilde x_3 \tilde x_2^\prime \leq \angle \tilde x_1 \tilde x_3 \tilde x_2 = \angle \tilde x_3
$, 
and hence $\angle x_3 \leq \angle \tilde x_3$.
A similar argument can be repeated for the other points $x_1$ and $x_2$.
\end{proof}

 \begin{figure}
   \centering
   \begin{tikzpicture}[mypoint/.style = {fill,inner sep=1.2pt,circle}, myarrow/.style={->,blue!80!black,shorten >= 1pt}, coordpoint/.style={inner sep=-1pt},angle/.style={red!50!black}]
 \node (x3) at (0,0) [mypoint,label=below:$x_3$]{};
 \node (x1) at (2.5,1.8) [mypoint,label=above:$x_1$]{};
 \node (x2) at (3,-.8) [mypoint,label=below:$x_2$]{};

 \draw (x3) to[out=20,in=-130] (x1) 
              node(abeta3a)[pos=0.21,coordpoint]{}
	      node(gamma1) [pos=0.5,sloped, above]{$\gamma_1$};
 \draw (x3) to[out=5,in=140] (x2) 
              node(abeta3b)[pos=0.2,coordpoint]{}
	      node(gamma2) [pos=0.5, sloped, below] {$\gamma_2$};
 \draw (x1) to[out=-120,in=135] (x2);

 \draw[angle] (abeta3a) to[out=-45,in=60] (abeta3b);

 \path (x3) +(13:13mm) node[angle]{$\beta_3$};

 \node at (0.5,1.5) {$\M$};

 \begin{scope}[xshift=3.7cm]
 \node (tx3) at (0,0) [mypoint,label=below:$\tilde x_3$]{};
 \node (tx1) at (3,2.2) [mypoint,label=above:$\tilde x_1$]{};
 \node (tx2) at (-16:3.7) [mypoint,label=below:$\tilde x_2$]{};
 \node (tx2p) at (20:3.7) [mypoint,label=below right:$\tilde x_2^\prime$]{};

 \draw (tx3) to (tx1) 
              node(tabeta3a)[pos=0.21,coordpoint]{}
              node(tabeta3pa)[pos=0.3,coordpoint]{}
	      node(tgamma1) [pos=0.5,sloped, above]{$\tilde\gamma_1$};
 \draw (tx3) to (tx2) 
              node(tabeta3b)[pos=0.2,coordpoint]{}
	      node(tgamma2) [pos=0.5, sloped, below] {$\tilde\gamma_2$};
 \draw[gray] (tx3) to (tx2p) 
              node(tabeta3pb)[pos=0.3,coordpoint]{}
	      node(tgamma2p) [pos=0.65, sloped, below] {$\tilde\gamma_2^\prime$};

 \draw (tx1) to (tx2);

 \draw[angle] (tabeta3a) to[out=-45,in=60] (tabeta3b);
 \draw[angle] (tabeta3pa) to[out=-45,in=110] (tabeta3pb);

 \path (tx3) +(1:13mm) node[angle]{$\tilde\beta_3$};
 \path (tx3) +(27:17mm) node[angle]{$\beta_3$};

 \node at (0.5,1.5) {$\M_\Delta$};
 \end{scope}
\end{tikzpicture}
   \vspace{-5mm}

   \caption{The geodesic triangles and hinges used for the proof of Theorem \ref{thm:ATCT2}}
\label{fig:ATCT2}
\end{figure}
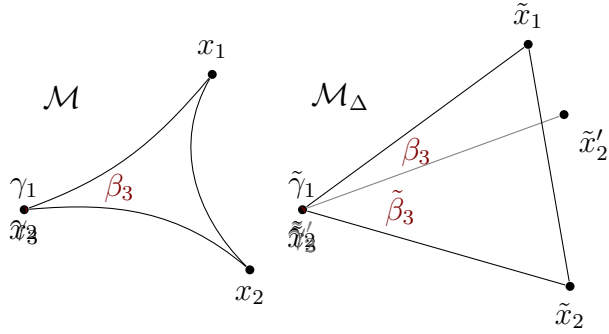

\myparagraph{Orthogonal decomposition of Jacobi fields}
Let $Y$ be a Jacobi field along a normal geodesic $\gamma$. 
The following Propositions shows that $Y$ can be decomposed in two orthogonal Jacobi fields.
\begin{proposition}\label{prop:Jacobidecomp}
  A Jacobi field $Y$ along a geodesic $\gamma$ can be decomposed as $Y=Y^\perp+Y^\parallel$, where $Y^\perp$ and $Y^\parallel$ are Jacobi fields which are, respectively, perpendicular and tangential to $\gamma$.
\end{proposition}
\begin{proof}
  The projection of $Y$ along $\dot{\gamma}$ is a function of the form $\metric{Y}{\dot{\gamma}}=at+b$, because
\begin{equation}
 \frac{\de^2}{\de t ^2} \metric{Y(t)}{\dot\gamma(t)}= \dert \metric{\nabla Y(t)}{\dot\gamma(t)}
 = \metric{\nabla\nabla Y(t)}{\dot\gamma(t)}= R(Y,\dot\gamma,\dot\gamma,\dot\gamma) =0.
\end{equation}
In the above we used, in succession, the metric compatibility property of $\nabla$, the definitions of geodesic and Jacobi field, and the properties of the curvature tensor.
 
The constants $a$ and $b$ can be determined using boundary conditions. Similar calculations show that $Y^\parallel=\metric{Y}{\dot{\gamma}}\dot{\gamma}$ is in fact a Jacobi field. It follows that $Y^\perp=Y-Y^\parallel$ is also a Jacobi field.
\end{proof}

\myparagraph{Comparison theorems for Jacobi fields}
We now review versions of the Rauch Comparison Theorems based on the presentation in \cite[pages 388--389]{Chavel:book06}.
\begin{thm}[Rauch Comparison Theorem I]\label{thm:RCT1}
 Let $X$ be a Jacobi field along and orthogonal to a normal geodesic $\gamma(s)$ satisfying $X(0)=0$ and without conjugate points. If the curvature is bounded above by $\Delta$, we have
 \begin{align}
  \metric{\nabla X}{X}&\geq \frac{C_\Delta}{S_\Delta} \norm{X}^2 & 
  \norm{\nabla X(0)}&\leq \frac{\norm{X}}{S_\Delta}\label{eq:rauch}
 \end{align}
\end{thm}
For the proof we will need the following Lemma \cite[pag. 387]{Chavel:book06}
\begin{lemma}\label{lm:derJacobi}
 Let $X(t)$ be a vector field along a geodesic $\gamma: t \in [0,\beta] \to \M$. If $X(0)$ and $\nabla X(0)$ are linearly dependent, or if $X(0)=0$, then $\dert\norm{X}(0)=\norm{\nabla X(0)}$.
\end{lemma}
\begin{proof}[Proof of Theorem~\ref{thm:RCT1}]
The proof is simply an adaptation of Theorem IX.2.1 in \cite{Chavel:book06} to our goals, where we identify $\eta=X$, $\psi=\norm{\nabla X(0)}S_\Delta$ and $\delta=\Delta$. In particular, that Theorem states that
$ 
 \dert\frac{\norm{\eta}}{\psi}=\frac{1}{\psi^2}\left(\frac{\de \norm{\eta}}{\de t}\psi - \eta\frac{\de \psi}{\de t}\right) \geq 0
$, 
which implies
\begin{equation}
 \frac{\de \norm{\eta}}{\de t}\psi - \eta\frac{\de\psi}{\de t}\geq 0 \implies \frac{\frac{\de \norm{\eta}}{\de t}}{\norm{\eta}}\geq \frac{\frac{\de \psi}{\de t}}{\psi}
\implies \frac{\metric{\nabla \eta}{\eta}}{\norm{\eta}^2} \geq \frac{C_\Delta}{S_\Delta}.
\end{equation}
With the above, the first equality of \eqref{eq:rauch} follows by Lemma~\ref{lm:derJacobi}.

The results in \cite{Chavel:book06} also state that $\norm{\eta}\geq\psi$, which is equivalent to the second part of \eqref{eq:rauch}. 
\end{proof}

\begin{thm}[Rauch Comparison Theorem II]\label{thm:RCT2}
 Let $X$ be a Jacobi field along and orthogonal to a normal geodesic $\gamma(s)$ satisfying $X(0)=0$ and without conjugate points. If the curvature is bounded below by $\delta$, we have
 \begin{align}
  \metric{\nabla X}{X}&\leq \frac{C_\delta}{S_\delta} \norm{X}^2 &
  \norm{\nabla X(0)}&\geq \frac{\norm{X}}{S_\delta}
 \end{align}
\end{thm}

\begin{proof}
  This Theorem is simply a restatement of Theorem IX.2.2 in \cite{Chavel:book06} with the identification $\eta=X$, $\psi=\norm{\nabla X(0)}S_\delta$ and $\kcurv=\delta$.
\end{proof}

\subsection{Derivative of the distance between two points on a geodesic hinge and proof of Lemma~\ref{thm:derdistancepair}}
\label{sc:firstderdistance}
This section is devoted to build results on the derivative of the distance between two points moving on the sides of a geodesic hinge, with the final goal of providing a proof for Lemma~\ref{thm:derdistancepair}. We will first obtain expressions in terms of angles between geodesics for general manifolds. 

Let $x_1$, $x_2\neq x_1$, and $y$ be three points in $\M$ such that $d_i=d(x_i,y)$ satisfies $0<d_i<r^\ast$, $i=1,2$, where $r^\ast$ is defined in \eqref{eq:rstar}.
Define the geodesic hinge $(y;\gamma_1,\gamma_2)$, where the sides are defined by the conditions $\gamma_1(0)=\gamma_2(0)=y$, $\gamma_1(1)=x_1$ and $\gamma_2(1)=x_2$. For each value of $t$, $0<t\leq 1$, define the minimal geodesic segment $\gamma_{12,t}(s)$ joining $\gamma_1(t)$ to $\gamma_2(t)$ (see Figure \ref{fig:geodtriang}). Note that, since $d_i<r^\ast$, $i=1,2$, by the triangular inequality we have that $d(x_1,x_2)<\inj_{x_1}\M$, therefore $\gamma_{12,t}(s)$ is uniquely defined (up to parametrization) for $t \in (0,1+\epsilon)$, where $\epsilon$ is small enough (so that $(1+\epsilon)d_i<r^\ast$, $i=1,2$). Denote the length of the geodesic segment $\gamma_{12,t}$ by $\phi_{12}(t)=L(\gamma_{12,t})$, which is nothing but the distance between $\gamma_1(t)$ and $\gamma_2(t)$ for a specific $t$. Our goal is to show that the derivative of $\phi^2_{12}$ is strictly positive on $t\in (0,1]$. Notice that $\gamma_{12,t}$ is defined for $t \in (0,1+\epsilon)$, hence the  derivative is well defined for $t=1$.

The first step is to obtain an expression for $\frac{\de \phi_{12}}{\de t}$.
\begin{proposition}\label{prop:dertriangdist}
 For a given $t_0 \in (0,1]$, consider the geodesic triangle $\triangle(y,\gamma_1(t_0),\gamma_2(t_0))$  and let $\beta_i$ be the angle at $\gamma_i(t_0)$ (see Figure \ref{fig:geodtriang}). Then 
$ 
\left.\frac{\de \phi_{12}}{\de t}\right|_{t=t_0}=d_1 \cos\beta_1+d_2\cos\beta_2
$. 
\end{proposition}

\begin{proof}
Let $d(x_1,x_2)$ be the distance function on $\M$.
By the definition of gradient we have
\begin{multline}
\left.\frac{\de \phi_{12}}{\de t}\right|_{t=t_0}= \metric{\grad_{x_1}d(\gamma_1(t_0),\gamma_2(t_0))}{\dot\gamma_1(t_0)}+\metric{\grad_{x_2}d(\gamma_1(t_0),\gamma_2(t_0))}{\dot\gamma_2(t_0)}\\
=\metric{\frac{-\log_{\gamma_1(t_0)}\gamma_2(t_0)}{\norm{\log_{\gamma_1(t_0)}\gamma_2(t_0)}}}{\dot\gamma_1(t_0)}
+ \metric{\frac{-\log_{\gamma_2(t_0)}\gamma_1(t_0)}{\norm{\log_{\gamma_2(t_0)}\gamma_1(t_0)}}}{\dot\gamma_2(t_0)}.
\end{multline}
Considering that $\norm{\dot\gamma_i(t)}=d_i$, the claim follows.
\end{proof}

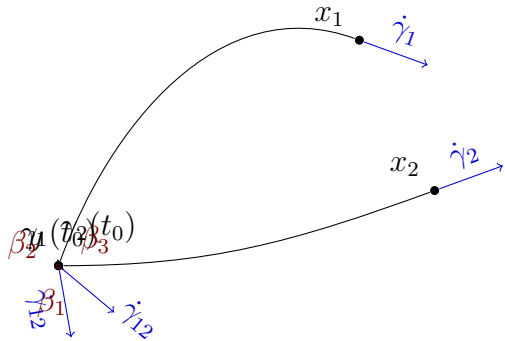
\begin{figure}[b]
\vspace{-2mm}
\centering
\begin{tikzpicture}[mypoint/.style = {fill,inner sep=1.2pt,circle}, myarrow/.style={->,blue!80!black,shorten >= 1pt}, coordpoint/.style={inner sep=-1pt},angle/.style={red!50!black}]
 \node (y) at (0,0) [mypoint,label=above left:$y$]{};
 \node (x1) at (4,3) [mypoint,label=above left:$x_1$]{};
 \node (x2) at (5,1) [mypoint,label=above left:$x_2$]{};

 \draw (y) to[out=70,in=160] (x1) 
	      node(abeta1a)[pos=0.6,coordpoint]{}
	      node(abetap1a)[pos=0.71,coordpoint]{}
	      node(gamma1) [pos=0.65,mypoint,label=above:$\gamma_1(t_0)$] {} ;
 \draw[label distance=1.5mm] (y) to[out=0,in=-160] (x2) 
	      node(abeta2a)[pos=0.6,coordpoint]{}
	      node(abetap2a)[pos=0.71,coordpoint]{}
	      node(gamma2) [pos=0.65,mypoint, label=above right:{$\!\!\!\!\gamma_2(t_0)$}] {};
 \draw (gamma1) to[out=-40, in=100] (gamma2)
	      node(abeta1b) [pos=0.06,coordpoint]{}
	      node(abeta2b) [pos=0.94,coordpoint]{}
	      node(abetap2b) [pos=1.1,coordpoint]{};
 \draw[myarrow] (x1) -- +(-20:1cm) node[pos=0.5,sloped,above]{$\dot\gamma_1$};
 \draw[myarrow] (x2) -- +(20:1cm) node[pos=0.5,sloped,above]{$\dot\gamma_2$};
 \draw[myarrow] (gamma1) -- +(-40:1cm) node[pos=1.3,sloped,above=-5pt]{$\dot\gamma_{12}$};
 \draw[myarrow] (gamma2) -- +(-80:1cm) node[pos=0.5,sloped,below]{$\dot\gamma_{12}$};

 \draw[angle] (abeta1a) to[out=-45,in=-135] (abeta1b);
 \draw[angle] (abeta2a) to[out=100,in=180] (abeta2b);

 \path (gamma1) +(-100:5mm) node[angle]{$\beta_1$};
 \path (gamma2) +(150:5.5mm) node[angle]{$\beta_2$};
 \path (y) +(35:6mm) node[angle]{$\beta_3$};
\end{tikzpicture}
\caption{The geodesic triangle used to study the derivative of the distance between $\gamma_1(t)$ and $\gamma_2(t)$}
\label{fig:geodtriang}
\end{figure}

The next step is to consider the particular case of manifolds with constant curvature $\Delta\geq0$ (for our purposes, the case $\Delta<0$ will be covered by the case $\Delta=0$). 
We have the following.
\begin{proposition}\label{prop:posderMD}
Let $\M$ be of constant curvature $\Delta\geq0$. Using the same definitions given at the beginning of the section, we have $\frac{\de \phi_{12}}{\de t}(t)>0$ for $t\in(0,1]$.
\end{proposition}
\begin{proof}
Let $l_i(t)=d_it$, $i=1,2$. In the case $\Delta=0$, from the cosine law we have $\phi_{12}(t)=\abs{t}\sqrt{d_1^2+d_2^2-2d_1d_2\cos\beta_3}$. The claim then easily follows.
For the case $\Delta>0$, as argued before, the triangular inequality implies $d(\gamma_1(t),\gamma_2(t))<\frac{\pi}{\sqd}$. In turn, this means that $\sin(\sqd \phi_{12})>0$. Instead of the derivative of $\phi_{12}(t)$, it will be convenient to use the derivative 
$ 
\dert\cos(\sqd \phi_{12})=-\sin(\sqd \phi_{12})\frac{\de \phi_{12}}{\de t}
$. 
From the above, $\frac{\de \phi_{12}}{\de t}>0$ if and only if $-\dert\cos(\sqd \phi_{12})>0$, hence the two expressions are equivalent for our purposes.

Using the cosine law for $\Delta>0$, we get
\begin{multline}
\label{eq:dertcos}
-\dert\cos(\sqd \phi_{12})
= -\dert (\cd(l_1(t))\cd(l_2(t))+\sd(l_1(t))\sd(l_2(t))\cos\beta_i) \\
= \sqd\biggl(\bigl(d_1 - d_2 \cos(\alpha)\bigr) \sd(d_1 t) \cd(d_2 t)
 + \bigl(d_2 - d_1 \cos(\alpha)\bigl) \cd(d_1 t) \sd(d_2 t)\biggr)
\end{multline}

Assume, without loss of generality, $d_1>d_2$ (if not, just swap the indexes throughout the proof) and recall $0<t<\frac{\pi}{2 \sqd d_1}$. This implies that $\sd(d_i t), \cd(d_i t)>0$ for $i=1,2$. Now, the condition $-\dert\cos(\sqd \phi_{12})>0$ can be manipulated as follows:
\begin{equation}
\bigl(d_1 - d_2 \cos(\alpha)\bigr) \sd(d_1 t) \cd(d_2 t)
 + \bigl(d_2 - d_1 \cos(\alpha)\bigl) \cd(d_1 t) \sd(d_2 t) > 0
\end{equation}
\begin{equation}
\frac{\sd(d_1 t) \cd(d_2 t)}{\cd(d_1 t) \sd(d_2 t)} > \frac{d_1 \cos(\alpha) - d_2}{d_1 - d_2 \cos(\alpha)}
\end{equation}
At this point, note that the RHS is always less or equal to one. Therefore, sufficient conditions for $-\dert\cos(\sqd \phi_{12})>0$ are given by
\begin{equation}
\frac{\sd(d_1 t) \cd(d_2 t)}{\cd(d_1 t) \sd(d_2 t)} > 1 \implies
\frac{\sd(d_1 t)}{\cd(d_1 t)} > \frac{\sd(d_2 t)}{\cd(d_2 t)} \implies
\tan(\sqd d_1 t) > \tan(\sqd d_2 t).
\end{equation}
Due to the monotonicity properties of the $\tan$ function, this condition is always satisfied under the assumptions that we made before, \ie $d_2<d_1<r^\ast$ and $t \in (0, 1]$. In other words, $-\dert\cos(\sqd \phi_{12})>0$, and therefore $\frac{\de \phi_{12}}{\de t}>0$ and the claim follows.
\end{proof}

We have now all the elements necessary to prove Lemma~\ref{thm:derdistancepair}.

\begin{proof}[Proof of Lemma \ref{thm:derdistancepair}]
 We first consider the case where the three points are all distinct. Notice that showing $\dert \frac{\phi_{12}^2}{2}=\phi_{12}\frac{\de \phi_{12}}{\de t}>0$ is equivalent to showing $\frac{\de \phi_{12}}{\de t}>0$. For any $t_0\in(0,1]$ consider the geodesic triangle $T=\triangle \bigl(y,\gamma_1(t_0),\gamma_2(t_0)\bigr)$. Build a triangle $T_\Delta=\triangle (\tilde y, \tilde x_1, \tilde x_2)$ in $\M_\Delta$ having the same side lengths as $T$. Define the geodesics $\tilde \gamma_i(t): t \to \M_\Delta$ such that $\tilde \gamma_i(0)=\tilde y$ and $\tilde \gamma_i(t_0)=\tilde x_i$, $i=1,2$. Define also $\tilde \phi_{12}(t)=d\bigl(\gamma_1(t),\gamma_2(t)\bigr)$. Let $\beta_i=\angle \gamma_i(t_0)$ and $\tilde \beta_i=\angle \tilde\gamma_i(t_0)$, $i=1,2$. According to Theorem \ref{thm:ATCT2}, $\beta_i\leq \tilde \beta_i$, $i=1,2$. Using Proposition \ref{prop:dertriangdist} and Proposition \ref{prop:posderMD} (if $\Delta<0$, use $\Delta=0$), this implies $\frac{\de \phi_{12}}{\de t}\geq\frac{\de \tilde\phi_{12}}{\de t}>0$, and the claim is shown.
 Next, consider the case $x_2=y$, $x_1\neq x_2$. Then $\phi_{12}=d^2(x_1,y)t$ and the claim can be shown by direct computation. The same applies by swapping the roles of $x_1$ and $x_2$.
 Finally, if $x_1=x_2=y$, then $\phi_{12}\equiv 0$ and the claim is trivial.
\end{proof}

\subsection{Bounds on the Hessian of the squared distance between two points}
\label{sc:secondderdistancegen}
In this section we compute and give bounds on the second derivative of the distance (and distance squared) between two points moving on geodesics. First, we derive a general expression that depends only on the relative velocities and angles between geodesics. Then, 
 we compute concrete bounds for the case of manifolds with bounded sectional curvature. We refer to \cite{Tron:TR11} for the case of manifolds with constant curvature. From the bounds on the second derivative, we can then obtain the bound on the Hessian of the squared distance, which is used in Theorem~\ref{thm:boundhessiandistance}.

\subsubsection{The general case}
\label{sc:secondderdistance}
Define two geodesics $\gamma_1, \gamma_2 : (-\epsilon,\epsilon) \to \M$ such that there exist a minimal geodesic $\gamma_{12,t}(s)$ joining $\gamma_1(t)$ to $\gamma_2(t)$ for all $t \in (-\epsilon,\epsilon)$. Using the same notation as in Section \ref{sc:firstderdistance} of this Appendix, we denote the length of the geodesic segment $\gamma_{12,t}$ by $\phi_{12}(t)=L(\gamma_{12,t})$, which is nothing but the distance between $\gamma_1(t)$ and $\gamma_2(t)$ for a specific $t$. In this section we will find bounds on the second derivative of $\phi_{12}(t)$ around $t_0=0$.

Define the geodesic variation $\alpha: [0,1] \times [a,b] \to \M$, such that the map $s \mapsto \alpha(t_0,s)$ traces the geodesic $\gamma_{12,t_0}(s)$, $\norm{\dot\gamma_{12,t}}=1$. Define $\partial_s\alpha=(D\alpha)\partial_s$ and $\partial_t\alpha=(D\alpha)\partial_t$, where $\partial_s$ (resp., $\partial_t$) denotes the partial derivation operator with respect to the variable $s$ (resp., $t$). Since $\alpha(t,s)$ traces geodesics, the vector field $X(s)=\partial_s \alpha |_{t=t_0}$ is a Jacobi field along $\gamma_{12,t}$ \cite[page 36]{Sakai:book96}. 

We have the following Theorem for computing the second derivative of the distance.
\begin{thm} \label{thm:secondder}
 Using the notation above, we have
\begin{equation}
\label{eq:secondderdistance}
 \left.\frac{\de^2 \phi_{12}}{\de t^2}\right|_{t=t_0}= \left.\metric{\nabla X(s)^\perp}{X(s)^\perp}\right|_0^l,
\end{equation}
and
\begin{equation}\label{eq:secondderdistance2}
\left.\frac{\de^2}{\de t^2} \frac{\phi_{12}^2}{2}\right|_{t=t_0}
= \left(\left.\metric{X(s)}{\frac{\gamma_{12,t_0}(s)}{\norm{\gamma_{12,t_0}(s)}}}\right|_0^l\right)^2 + l\left.\metric{\nabla X(s)^\perp}{X(s)^\perp}\right|_0^l,
\end{equation}
where $l=\phi_{12}(0)<2r^\ast$.
\end{thm}
\begin{proof}
From \cite[page 76]{Chavel:book06} we get
\begin{equation}
   \left.\frac{\de^2 \phi_{12}}{\de t^2}\right|_{t=t_0}= \left.\metric{\nabla_{\partial_t} X}{\dot\gamma_{12,t}}\right|_a^b +\int_a^b\left(\norm{\nabla X ^\perp}^2- R(\dot\gamma_{12,t},X^\perp,\dot\gamma_{12,t},X^\perp)\right) \de s
\end{equation}
Then, notice that since $X^\perp$ is a Jacobi field (Proposition~\ref{prop:Jacobidecomp}), we have $- R(\dot\gamma_{12,t},X^\perp,\dot\gamma_{12,t},X^\perp) = \metric{\nabla\nabla X^\perp}{X^\perp}= \dert \metric{\nabla X^\perp}{X^\perp}-\metric{\nabla X^\perp}{\nabla X^\perp} = \dert \metric{\nabla X^\perp}{X^\perp}-\norm{\nabla X^\perp}^2$. Hence, we have
\begin{equation}
  \left.\frac{\de^2 \phi_{12}}{\de t^2}\right|_{t=t_0} \!\!\!\!=\!\! \int_a^b \dert \metric{\nabla X^\perp}{X^\perp} \de s= \left.\metric{\nabla X(s)^\perp}{X(s)^\perp}\right|_0^l,
\end{equation}
which is \eqref{eq:secondderdistance}. Equation \eqref{eq:secondderdistance2} follows from the fact that $\left.\frac{\de^2}{\de t^2} \frac{\phi_{12}^2}{2}\right|_{t=t_0}=\left(\left.\frac{\de \phi_{12}}{\de t}\right|_{t=t_0}\right)^2+\left.\phi_{12}\frac{\de^2 \phi_{12}}{\de t^2}\right|_{t=t_0}$.
\end{proof}

\emph{Remark:} Notice that the second derivative of $\phi_{12}$ depends only on the orthogonal component of the Jacobi field, $X^\perp$. Therefore, any two pairs of geodesics having the same $X^\perp(0)=\dot\gamma_1(t_0)^\perp$ and $X^\perp(l)=\dot\gamma_2(t_0)^\perp$, will have the same orthogonal Jacobi field component and will yeld the same second derivative of the distance $\phi_{12}$. However, the tangential components of $\dot\gamma_1(t_0)$ and $\dot\gamma_2(t_0)$ play a role in the second derivative of the \emph{squared} distance.

\subsubsection{Manifolds with bounded curvature}
In this section we will give bounds on the second derivative of the squared distance function given in \eqref{eq:secondderdistance2} for more general Riemannian manifolds in terms of the curvature bounds $\Delta$ and $\delta$. In particular, we will show the following result.
\begin{thm}\label{thm:secondderboundcurv}
 Define two geodesics $\gamma_1, \gamma_2 : (-\epsilon,\epsilon) \to \M$ such that $\gamma_2(t) \in \calB_\M(\gamma_1(t),\inj\M)$ for all $t \in (-\epsilon,\epsilon)$. Let $\phi_{12}(t)=d(\gamma_1(t),\gamma_2(t))$ and define $l=\phi_{12}(0)$. Then
 \begin{equation}\label{eq::secondderboundcurv1}
  \left.\frac{\de^2 \phi_{12}}{\de t^2}\right|_{t=t_0} \leq l\left(\frac{C_\delta(l)}{S_\delta(l)}+\frac{1}{S_\Delta(l)}\right)\bigl(\norm{\dot{\gamma}_1^\perp(0)}^2+\norm{\dot{\gamma}_2^\perp(0)}^2\bigr)
 \end{equation}
and
 \begin{equation}\label{eq::secondderboundcurv2}
  \left.\frac{\de^2 \phi_{12}}{\de t^2}\right|_{t=t_0} \leq \mu^d_{max}(l)\bigl(\norm{\dot{\gamma}_1(0)}^2+\norm{\dot{\gamma}_2(0)}^2\bigr),
 \end{equation}
where $\mu_{max}^d(l)=\max\{2,l\left(\frac{C_\delta(l)}{S_\delta(l)}+\frac{1}{S_\Delta(l)}\right)\}$.
\end{thm}
Also, by Definition \ref{def:hessianRiemannian}, $\mu_{max}(l)$ is a bound on the Hessian of the squared distance evaluated at $\bigl(\gamma_1(0),\gamma_2(0)\bigr)$. In addition, these bounds are sharp, in the sense that if $\delta=\Delta=\kappa$, we obtain the same bounds from constant curvature case \cite{Tron:TR11}.
\begin{proof}
We start from \eqref{eq:secondderdistance}. 
Since it is not easy to give a simple close form expression of $\nabla X$ in terms of $X$, we will give a way to bound each one of the terms $\metric{\nabla X^\perp}{X^\perp}$ at $s=0$ and $s=l$. We will decompose the Jacobi field $X^\perp$ in two components as $X^\perp=X_1+X_2$ where $X_1$ and $X_2$ are Jacobi fields satisfying the conditions $X_1(0)=0$, $X_1(l)=X(l)^\perp$, $X_2(0)=X(0)^\perp$ and $X_2(l)=0$. The main reason to do this is that $X_1$ and $X_2$ have now the property of vanishing at one of the endpoints, and we can therefore exploit results from standard Riemannian geometry texts. Note that since $\gamma_2(t) \in \calB_\M(\gamma_1(t),\inj\M)$, we have $l<2r^\ast$ and, by the Morse-Sch\"onberg Theorem \cite[p. 86]{Chavel:book06} $X_1$ and $X_2$ have no conjugate points on $\gamma_{12}$. We can therefore apply the Rauch comparison theorems of \S\ref{sec:general-results} of this Appendix to get bounds on these Jacobi fields.

More concretely, the second derivative of the distance is given by
\begin{multline}
  \left.\frac{\de^2 \phi_{12}}{\de t^2}\right|_{t=t_0}= \left.\metric{\nabla X(s)^\perp}{X(s)^\perp}\right|_0^l
= \metric{\nabla X(l)^\perp}{X(l)^\perp}-\metric{\nabla X(0)^\perp}{X(0)^\perp}\\
=\metric{\nabla X_1(l)}{X_1(l)}+\metric{\nabla X_2(l)}{X_1(l)}
-\metric{\nabla X_1(0)}{X_2(0)}-\metric{\nabla X_2(0)}{X_2(0)}\label{eq:secondderinner}
\end{multline}

Using Theorem \ref{thm:RCT2} (Rauch Comparison Theorem II) we have
\begin{align}
 \metric{\nabla X_1(l)}{X_1(l)}\leq \frac{C_\delta(l)}{S_\delta(l)} \norm{X(l)^\perp}^2, &&
 \metric{-\nabla X_2(0)}{X_2(0)}\leq \frac{C_\delta(l)}{S_\delta(l)} \norm{X(0)^\perp}^2. \label{eq:innerproductder}
\end{align}
Note that for $X_2$, in order to apply Theorem Theorem \ref{thm:RCT2}, we need to reverse the parametrization of $\gamma_{12,t_0}(s)$ as $s^\prime=l-s$. This has the effect that $\nabla X_2(s^\prime)|_{s^\prime=l}=\nabla_{-\dot\gamma_{12,t_0}(s)} X_2(s)|_{s=0}=-\nabla X_2(0)$. This explains the negative sign in the second inequality of \eqref{eq:innerproductder}.

Using the Cauchy-Schwarz inequality, Theorem \ref{thm:RCT1} (Rauch Comparison Theorem I) and the inequality $ab\leq \frac{a^2+b^2}{2}$ we have
\begin{align}
 \metric{\nabla X_2}{X_1^\perp(l)}&
  \leq \frac{1}{S_\Delta(l)} \norm{X_2(0)}\norm{X_1(l)}
  \leq \frac{1}{2 S_\Delta(l)} (\norm{X(0)^\perp}^2+\norm{X(l)^\perp}^2) \label{eq:innerproductder2}
\end{align}
\begin{align}
-\metric{\nabla X_1(0)}{X_2(0)}&
  \leq \frac{1}{S_\Delta(l)} \norm{X_1(l)}\norm{X_2(0)} 
  \leq \frac{1}{2 S_\Delta(l)}(\norm{X(0)^\perp}^2+\norm{X(l)^\perp}^2) \label{eq:innerproductder3}
\end{align}

Combining \eqref{eq:innerproductder}, \eqref{eq:innerproductder2} and \eqref{eq:innerproductder3} into \eqref{eq:secondderinner}, we get
\begin{equation}
 \left.\frac{\de^2 \phi_{12}}{\de t^2}\right|_{t=t_0} \leq \left(\frac{C_\delta(l)}{S_\delta(l)}+\frac{1}{S_\Delta(l)}\right)(\norm{X(0)^\perp}^2+\norm{X(l)^\perp}^2),
\end{equation}
which is equivalent to \eqref{eq::secondderboundcurv1}. Combining this with \eqref{eq:secondderdistance} we obtain the equivalent of \eqref{eq::secondderboundcurv2}:
\begin{multline}
 \left.\frac{\de^2}{\de t^2}\frac{\phi_{12}^2}{2}\right|_{t=t_0}
\leq(\norm{X(0)^\parallel}+\norm{X(l)^\parallel})^2+l\left(\frac{C_\delta(l)}{S_\delta(l)}+\frac{1}{S_\Delta(l)}\right)(\norm{X(0)^\perp}^2+\norm{X(l)^\perp}^2) \\
 \leq \mu_{max}^d(l)\bigl(\norm{X(0)}^2+\norm{X(l)}^2\bigr).
\end{multline}
\end{proof}

\end{appendix}




\ifCLASSOPTIONcaptionsoff
  \newpage
\fi


\begin{thebibliography}{10}

\bibitem{Tron:ICDSC08}
R.~Tron, R.~Vidal, and A.~Terzis,
\newblock ``Distributed pose averaging in camera networks via consensus on
  {$SE(3)$},''
\newblock in {\em International Conference on Distributed Smart Cameras}, 2008.

\bibitem{Tron:CDC09}
R.~Tron and R.~Vidal,
\newblock ``Distributed image-based {3-D} localization in camera sensor
  networks,''
\newblock in {\em Conference on Decision and Control}, 2009.

\bibitem{Sarlette:TAC10}
A.~Sarlette, S.~Bonnabel, and R.~Sepulchre,
\newblock ``Coordinated motion design on {L}ie groups,''
\newblock {\em {IEEE} Transactions on Automatic Control}, vol. to be published,
  2010.

\bibitem{Olfati:CDC06}
R.~Olfati-Saber,
\newblock ``Swarms on sphere: A programmable swarm with synchronous behaviors
  like oscillator networks,''
\newblock in {\em {IEEE} Conference on Decision and Control}, 2006, pp.
  5060--5066.

\bibitem{Scardovi:SCL07}
L.~Scardovi, A.~Sarlette, and R.~Sepulchre,
\newblock ``Synchronization and balancing on the {$N$}-torus,''
\newblock {\em Systems and Control Letters}, vol. 56, no. 5, pp. 335--341,
  2007.

\bibitem{Sarlette:SJOC2009}
A.~Sarlette and R.~Sepulchre,
\newblock ``Consensus optimization on manifolds,''
\newblock {\em SIAM J. Control and Optimization}, vol. 48, no. 1, pp. 56--76,
  2009.

\bibitem{Hatanaka:CDC10}
T.~Hatanaka, M.~Fujita, and F.~Bullo,
\newblock ``Vision-based cooperative estimation via multi-agent optimization,''
\newblock in {\em {IEEE} Conference on Decision and Control}, 2010.

\bibitem{Igarashi:TraCST09}
Y.~Igarashi, T.~Hatanaka, M.~Fujita, and M.W. Spong,
\newblock ``Passivity-based attitude synchronization in {$SE(3)$},''
\newblock {\em IEEE Transactions on Control Systems Technology}, vol. 17, no.
  5, pp. 1119 --1134, 2009.

\bibitem{Olshevsky:SOC09}
A.~Olshevsky and J.~Tsitsiklis,
\newblock ``Convergence speed in distributed consensus and averaging,''
\newblock {\em SIAM Journal of Control and Optimization}, vol. 48, no. 1, pp.
  33--55, 2007.

\bibitem{Olfati:TAC04}
R.~Olfati-Saber and R.~Murray,
\newblock ``Consensus problems in networks of agents with switching topology
  and time-delays,''
\newblock {\em {IEEE} Transactions on Automatic Control}, vol. 49, no. 3, pp.
  1520--1533, 2004.

\bibitem{Sakai:book96}
T.~Sakai,
\newblock {\em {R}iemannian Geometry}, vol. 149 of {\em Translations of
  Mathematical Monographs},
\newblock {A}merican {M}athematical {S}ociety, 1996.

\bibitem{Chavel:book06}
I.~Chavel,
\newblock {\em {R}imeannian Geometry: a Modern Introduction}, vol.~98 of {\em
  {C}ambridge studies in advanced mathematics},
\newblock {C}ambridge {U}niversity {P}ress, 2 edition, 2006.

\bibitem{DoCarmo:riemannian92}
M.~P. do~Carmo,
\newblock {\em Riemannian geometry},
\newblock Birkh\"auser, Boston, MA, 1992.

\bibitem{Absil:book08}
P.-A. Absil, R.~Mahony, and R.~Sepulchre,
\newblock {\em Optimization Algorithms on Matrix Manifolds},
\newblock Princeton University Press, Princeton, NJ, 2008.

\bibitem{Edelman98}
A.~Edelman, T.~Arias, and S.~T. Smith,
\newblock ``The geometry of algorithms with orthogonality constraints,''
\newblock {\em SIAM Journal of Matrix Analysis Applications}, vol. 20, no. 2,
  pp. 303--353, 1998.

\bibitem{Tron:TR11}
R.~Tron, B.~Afsari, and R.~Vidal,
\newblock ``Average consensus on {R}iemannian manifolds with bounded
  curvature,''
\newblock Tech. {R}ep., Johns Hopkins University, 2011.

\bibitem{Udriste:book94}
C.~Udriste,
\newblock ``Convex functions and optimization methods on {R}iemannian
  manifolds,''
\newblock in {\em Mathematics and Applications}. 1994, vol. 297, Kluwer
  Academic.

\bibitem{BoydVandenberghe04}
S.~Boyd and L.~Vandenberghe,
\newblock {\em Convex Optimization},
\newblock Cambridge University Press, 2004.

\bibitem{Afsari:ProcAMS11}
B.~Afsari,
\newblock ``{R}iemannian {$L^{p}$} center of mass: Existence, uniqueness, and
  convexity,''
\newblock {\em Proceedings of the {AMS}}, vol. 139, no. 2, pp. 655--673, 2011.

\bibitem{Afsari:11}
R.~Vidal B.~Afsari, R.~Tron,
\newblock ``On the convergence of gradient descent for locating the
  {R}iemmanian center of mass,''
\newblock {\em Preprint on ArXiv}, 2011.

\bibitem{Cortes:Automatica08}
J.~Cort{\'e}s,
\newblock ``Distributed algorithms for reaching consensus on general
  functions,''
\newblock {\em Automatica}, vol. 44, no. 3, pp. 726--737, 2008.

\bibitem{Protter:book98}
M.~H. Protter,
\newblock {\em Basic Elements of Real Analysis},
\newblock Undergraduate Texts in Mathematics. Springer, 1998.

\end{thebibliography}
\end{document}